\documentclass[11pt]{article}
\usepackage{amsmath,amsthm,amsfonts,amssymb,amscd,amsxtra}
\usepackage{url}
\usepackage[margin=1in]{geometry}
\usepackage{color}
\usepackage{graphicx}
\usepackage{caption}
\usepackage{subcaption}
\usepackage{anyfontsize}
\usepackage{pgfplots}
\pgfplotsset{compat=1.17}
\usepackage{algpseudocode}
\usepackage{mathtools}
\usepackage{algorithm}
\usepackage[numbers]{natbib}
\usepackage[brazilian]{cleveref}
\usepackage{comment}

\usepackage{bbm}

\algnewcommand\Or{\textbf{or} }
\algnewcommand\And{\textbf{and} }

\theoremstyle{plain}
\newtheorem{theorem}{Theorem}[section]      
\newtheorem{proposition}[theorem]{Proposition}
\newtheorem{lemma}[theorem]{Lemma}

\newtheorem{definition}[theorem]{Definition}

\newtheorem{remark}[theorem]{Remark}
\newtheorem{assumption}[theorem]{Assumption}

\def\argmin{\operatorname{argmin}}

\DeclareMathOperator{\Id}{Id}
\DeclareMathOperator{\Ima}{Im}

\DeclareMathOperator{\Ker}{Ker}

\DeclareMathOperator{\inte}{int}

\newcommand{\K}{\mathcal K}

\newcommand{\C}{\mathcal C}

\newcommand{\LL}{\mathbb L}

\newcommand{\sym}{\mathbb S}
\newcommand{\R}{\mathbbm R}

\newcommand{\N}{\mathbbm N}
\newcommand{\X}{\mathbbm X}
\newcommand{\Y}{\mathbbm Y}
\newcommand{\Z}{\mathbbm Z}

\numberwithin{equation}{section}

\begin{document}

\title{A semi-smooth Newton method for the nonlinear conic problem with generalized simplicial cones}
\author{Nicolas F. Armijo \thanks{Department of Applied Mathematics, State University of Campinas, Brazil (e-mail: {\tt
nfarmijo@ime.unicamp.br}).  The author was supported by Fapesp grant 2019/13096-2.}  
\and
Yunier  Bello-Cruz\thanks{Northern Illinois University, USA (e-mail: {\tt
yunierbello@niu.edu}).  The author was partially supported by the NSF Grant DMS-2307328 and by
an internal grant from NIU.}
\and
Gabriel Haeser \thanks{Department of Applied Mathematics, University of São Paulo, Brazil (e-mail: {\tt
ghaeser@ime.usp.br}). The author was supported by CNPq and Fapesp grant 2018/24293-0.}
}
\date{\today}
\maketitle
\vspace{-0.4 in}
\begin{abstract}
In this work we develop and analyze a semi-smooth Newton method for the
general nonlinear conic programming problem. In particular, we study
the problem with a generalized simplicial cone, i.e., the image of a
symmetric cone under a linear mapping. We generalize Robinson's normal
equations to a conic setting, yielding what we call the conic
projection equations. The resulting system is equivalent to the KKT
conditions associated with the nonlinear conic programming problem. A
semi-smooth Newton iteration is proposed for solving it, and local
quadratic convergence is established. We study properties of
generalized simplicial cones and prove strong semi-smoothness of the
projection operator onto them. Numerical experiments compare the
method against a recent smoothing Newton approach on the circular cone
programming problem, and we also apply it to the low-rank matrix
completion problem.


\end{abstract}

\section{Introduction} \label{section_intro}

In this paper we study the general nonlinear conic programming (NCP)
problem:
\begin{equation}\label{gen_conic_nonlinconst}
\begin{array}{cl}
\min & f(x) \\
\text{s.t.} & g(x) \in \K,
\end{array}
\end{equation}
where $\X$ and $\Y$ are finite-dimensional inner product spaces,
$\K \subset \Y$ is a closed convex cone, and $f:\X \to \R$,
$g:\X \to \Y$ are twice continuously differentiable functions. The
nonlinear conic programming problem has been widely studied over the
last three decades, with significant contributions in the areas of
optimality conditions, constraint qualifications, and algorithms for
particular, yet important, instances. See \cite{bonnans,lobo}.

Our approach is based on what we call \emph{conic projection
equations}, which extend Robinson's normal equations
\cite{Robinson:1992}. Robinson formulated these equations in the
context of complementarity conditions for minimizing a function $f$
over a closed convex set $C$; they served as the foundation for
algorithms targeting variational inequalities, see, for example,
\cite{Qi:2000} for box-constrained problems and \cite{Zhou:2014} for
polyhedral convex set constrained problems.

In the nonlinear conic programming setting, analogous projection-based
systems were studied in \cite{Bello:2016_4}, where the authors
addressed the quadratic conic programming problem in the particular
case of the nonnegative cone. They employed a piecewise linear system to locate KKT points and a
semi-smooth Newton method to solve it; a related approach was developed
in \cite{Barrios:2016_4} for a special piecewise linear system arising
from positively constrained convex quadratic programming. This
formulation is also closely related to the absolute value equation
(AVE); see, for example, \cite{Bello:2016_4}. A more recent
contribution appears in \cite{Armijo:2023_4}, where Gauss--Seidel-type
and Jacobi-type methods were proposed for solving such piecewise linear
systems. In the same vein, \cite{Bello:2017_4} extended the nonlinear
system to the second-order cone, addressing the quadratic second-order
cone programming problem via a semi-smooth Newton method with global
convergence guarantees under assumptions on the Hessian of the
objective function. For the convex quadratic semidefinite programming
problem, the authors of \cite{Li:2018_4} proposed a two-phase
augmented Lagrangian method, resulting in the MATLAB package QSDPNAL.
The earlier work \cite{Qi:2006_4} studied the same problem and, based
on an equivalent formulation, proposed a Newton method. More recently,
\cite{Armijo:2025_4} addressed the non-convex quadratic conic
programming problem, focusing on the nearest correlation matrix problem
and proposing a locally quadratically convergent semi-smooth Newton
method using a different, easy-to-compute subdifferential.

Over time, many algorithms have been developed for the various
structures arising in conic programming. Interior-point methods, such
as \cite{Helmberg:1996_4} for semidefinite programming, are classical
in the area. Their efficiency, particularly for linear programming, has
made them a sustained area of research; see, for example,
\cite{Gill:2024_4} for recent developments. Augmented Lagrangian (AL)
methods offer an alternative approach: they approximately minimize a
penalized version of the classical Lagrangian by solving a subproblem
involving the constraint function, followed by updates of the multiplier and
penalty parameters. Notable AL methods related to our work include
SDPNAL+ \cite{Yang:2015_4} and QSDPNAL \cite{Li:2018_4}, which have
demonstrated efficiency for linear and quadratic semidefinite
programming. Augmented Lagrangian methods remain an active area of
research, receiving significant attention in nonlinear programming,
conic programming \cite{Fukuda:2024_4,Lu:2024_4,He:2023_4}, and
Riemannian optimization \cite{Andreani:2024a_4}.

The methods most directly related to this work are Newton-type methods
\cite{Izmailov:2014_4}, which solve the primal-dual complementarity
optimality conditions by approximating the zeros of a nonlinear system
of equations. A well-known variant employs smoothing functions to
approximate the complementarity conditions, with the properties of
these functions ensuring that solutions of the resulting nonlinear
system correspond to KKT points of the original problem. Smoothing
methods have been proposed recently for second-order cone programming
\cite{Tang:2022_4,Tang:2023_4,Chi:2016_4} and semidefinite programming
\cite{Liang:2023_4}. Additionally, \cite{Im:2024_4} studied a
semi-smooth Newton method for the nearest correlation matrix problem,
and \cite{Zhao:2022_4} proposed a Lagrange--Newton algorithm for
sparse nonlinear programming.

In this paper, we extended Robinson's normal equations to the setting
of nonlinear conic programming, yielding a system of \emph{conic
projection equations} that is equivalent to the KKT conditions of
problem~\eqref{gen_conic_nonlinconst} (Theorem~\ref{equiv_kkt_eq}). In
contrast to smoothing-based approaches, our formulation does not employ
a smoothing function; instead, it encodes complementarity through a
direct nonlinear system involving the conic projection.

We study the generalized simplicial cone $M\K$, defined as the image
of a symmetric cone $\K$ under a linear mapping $M$, and establish the
strong semi-smoothness of the projection onto $M\K$
(Theorem~\ref{metr_proj_str_semi}), including the case where $M$ is
rank-deficient. We also derive an explicit characterization of the dual
cone $(M\K)^*$ (Proposition~\ref{proj_simpli_semismth}) and provide
sufficient conditions for the closedness of $M\K$
(Propositions~\ref{suff_cond_simpli_close_1}
and~\ref{suff_cond_simpli_close_2}). Building on these results, we
propose a semi-smooth Newton method
(Algorithm~\ref{alg:plucg_semismooth}) for solving the conic
projection equations, incorporating a regularization strategy for
singular Jacobians and an escape mechanism for non-strongly stationary
points. We prove local quadratic convergence (Theorem~\ref{thm_quad_conv}).

We evaluate the proposed method on two classes of problems: circular
cone programming and low-rank matrix completion. Second-order cone
programming (SOCP) has numerous applications in areas such as control,
signal processing, and finance, and its structure, variants, and
solution methods have been extensively studied; see, for example,
\cite{Chen:2025,Chen:2025_2,Garrido:2025,Lin:2025,SilveiraHaeser:2025,Zheng:2024}.
Our focus is on circular cone programming, a natural generalization of
SOCP, and specifically on Newton-type methods, in particular the
smoothing method of~\cite{Tang:2022_4}. The semi-smooth Newton method
developed here achieves comparable or better precision than
\cite{Tang:2022_4}, while being up to two orders of magnitude faster.

The low-rank matrix completion problem is intrinsically challenging due
to the combinatorial nature of the rank constraint, which significantly
complicates the design of high-precision algorithms. Our interest in
this problem was motivated by the reformulation
of~\cite{Bertsimas:2021_4}, which replaces the rank constraint with
continuous matricial constraints, enabling a purely continuous
optimization approach. Leveraging this reformulation, we apply the
proposed semi-smooth Newton method and compute high-precision KKT
points for more than $94\%$ of the tested instances.

The paper is organized as follows. In Section~\ref{sec:preliminaries3},
we present preliminary results, including properties of generalized
simplicial cones. Section~\ref{section_ncp_problem} reviews the
nonlinear conic programming problem and introduces the generalized
system of projection equations, establishing its equivalence with the
KKT conditions. In Section~\ref{section_semismoothNewt}, we develop the
semi-smooth Newton method and analyze its convergence properties.
Finally, in Section~\ref{section_num_exps}, we report numerical
experiments on circular cone programming and low-rank matrix
completion.

\section{Preliminaries} \label{sec:preliminaries3}

Let $\X$ and $\Y$ be finite-dimensional normed vector spaces equipped with an inner product, which we indistinctly denote by $\langle \cdot, \cdot \rangle$. Let $\Id$ denote the identity operator. Given a linear operator $M : \X \rightarrow \Y$, its operator norm is defined by $\|M\| := \sup\{\|Mx\| : x \in \X,\, \|x\|=1 \}$, and its adjoint $M^* : \Y \rightarrow \X$ is the unique linear operator satisfying $\langle Mx, y \rangle = \langle x, M^* y \rangle$ for all $x \in \X$ and $y \in \Y$. We say that $M$ is positive semidefinite (resp.\ positive definite) if $\langle Mx, x \rangle \geq 0$ (resp.\ $> 0$) for all $x \in \X$. We denote by $\sym^n$ the set of $n \times n$ real symmetric matrices, by $\sym^n_+$ the subset of positive semidefinite matrices, and by $\LL^n$ the second-order cone. For a fixed cone $\K \subset \X$, the dual and polar cones are defined by
$$\K^* := \{y \in \X : \langle y, x \rangle \geq 0,\; \forall\, x \in \K\}
\quad \text{and} \quad
\K^{\circ} := \{y \in \X : \langle y, x \rangle \leq 0,\; \forall\, x \in \K\} = -\K^*,$$
respectively.

\begin{lemma}[Spectral properties of symmetric matrices; 
  see e.g.\ {\cite[Ch.~4]{Horn:2013_4}}]\label{eigen_mat}
Let $E \in \sym^n$ and let $\lambda_{\rm min}(E)$, $\lambda_{\max}(E)$ denote the smallest and largest eigenvalues of $E$, respectively. Then:
\begin{enumerate}
\item $x^T E x \leq \lambda_{\max}(E)\, x^T x$ for all $x \in \R^n$.
\item If $E$ is positive definite, then $\lambda_{\max}(E^{-1}) = 1 / \lambda_{\rm min}(E)>0$.
\end{enumerate}
\end{lemma}

The projection of a point $x$ onto a closed convex set $\K$ is defined by $\Pi_{\K}(x) := \argmin\{\|y - x\| : y \in \K\}$. For a function $f$, we denote by $D_f$ the set of points where $f$ is differentiable, by $Df(x)$ its differential at $x$, and by $f'(x)$ its Jacobian. The Clarke generalized Jacobian is denoted by $\partial_C f(x)$ and defined as
$$\partial_C f(x) := \operatorname{conv}\left\{\lim_{k \to \infty} f'(x_k) : x_k \to x,\; x_k \in D_f \right\}.$$
Throughout this work, we write $V_{\K}(x)$ for a generic element of $\partial_C \Pi_{\K}(x)$ and $V_f(x)$ for a generic element of $\partial_C f(x)$.

A key notion for our analysis is the strong semi-smoothness of a function. A Lipschitz function $f : \X \to \Y$ is said to be \emph{strongly semi-smooth} at $x$ if
$$f(x+h) - f(x) - V_f(x+h)\, h = O(\|h\|^2)$$
for all $h \in \X$ and all $V_f(x+h) \in \partial_C f(x+h)$. We say that $f$ is strongly semi-smooth if it is strongly semi-smooth at every $x \in \X$.

The following two classical results are used in our analysis.

\begin{theorem}[Contraction mapping principle {\cite[Thm.~8.2.2, p.~153]{Ortega:1987_4}}] \label{fixedpoint3}
Let $\Phi : \R^n \to \R^n$ be a function. Suppose there exists $\lambda \in [0,1)$ such that $\|\Phi(y) - \Phi(x)\| \leq \lambda\, \|y - x\|$ for all $x, y \in \R^n$. Then there exists a unique $\bar{x} \in \R^n$ such that $\Phi(\bar{x}) = \bar{x}$.
\end{theorem}

\begin{theorem}[Mean value theorem {\cite[Prop.~2.6.5, p.~72]{Clarke:1990_4}}] \label{meanval_theo3}
Let $\X, \Y$ be finite-dimensional normed spaces and let $f : \X \rightarrow \Y$ be a Lipschitz function. Then,
$$f(y) - f(x) \in \operatorname{conv}\!\left(\partial_C f([x,y])\cdot(y-x)\right),$$
that is, $f(y) - f(x) = h$ for some $h$ in the convex hull of the set $\{V(y-x) : V \in \partial_C f(z),\; z = tx + (1-t)y,\; t \in [0,1]\}$.
\end{theorem}

We conclude this section by presenting key properties of the projection onto a closed convex cone.

\begin{theorem}[Properties of the conic projection and its generalized Jacobian; 
  see e.g.\ {\cite[Thm.~2.1]{Armijo:2025_4}}]\label{exis_diff_proj2}
Let $\X$ be a finite-dimensional normed vector space and let $\K \subset \X$ be a closed convex cone. Then the projection $\Pi_{\K}$ is differentiable almost everywhere, and both the Jacobian $\Pi'_{\K}(x)$ and any element $V_{\K}(x) \in \partial_C \Pi_{\K}(x)$ is a self-adjoint positive semidefinite linear operator satisfying the following properties:
\begin{enumerate}
\item \label{norm_diff122} $\|V_{\K}(x)\| \leq 1$ for all $V_{\K}(x) \in \partial_C \Pi_{\K}(x)$ and all $x \in \X$.
\item \label{Pxx=Px2} $\Pi_{\K}'(x)\, x = \Pi_{\K}(x)$ for all $x \in D_{\Pi_{\K}}$.
\item \label{Vxx=Px2} $V_{\K}(x)\, x = \Pi_{\K}(x)$ for all $V_{\K}(x) \in \partial_C \Pi_{\K}(x)$.
\item \label{eigen_gendiff2} The eigenvalues of $\Pi'_{\K}(x)$ satisfy
$$ 0 \leq \lambda_{\rm min}(\Pi'_{\K}(x)) \leq \lambda_{\max}(\Pi'_{\K}(x)) \leq 1.$$
Since eigenvalues are continuous functions of the matrix entries, the same bounds hold for any $V_{\K}(x) \in \partial_C \Pi_{\K}(x)$:
$$ 0 \leq \lambda_{\rm min}(V_{\K}(x)) \leq \lambda_{\max}(V_{\K}(x)) \leq 1.$$
\end{enumerate}
\end{theorem}
\begin{lemma}[Linearization error bound for the conic projection; 
  {\cite[Lem.~2.1]{Armijo:2025_4}}] \label{norm_diff3}
Let $x, y \in \X$ and $V_{\K}(x) \in \partial_C \Pi_{\K}(x)$. Then \[\|\Pi_{\K}(y) - \Pi_{\K}(x) - V_{\K}(x)(y-x)\| \leq \|y - x\|.\]
\end{lemma}

\subsection{Generalized simplicial cones} \label{section_gensimpli}

\hspace{\parindent} In this section we study properties of the generalized simplicial cone, that is, cones in the form $M\K$, where $M\colon \X \to \Y$ is a linear mapping and $\K \subset \X$ is a symmetric cone. A simplicial cone is defined as a cone in which every generator consists solely of extreme rays. However, while finitely generated simplicial cones are typically represented as $M\R^n_+$, we use the term ``generalized'' here in the sense that we study a general cone $\K$ instead of $\R^n_+$. One motivation for this study is the circular cone programming problem, where the cone constraint is defined by a linear transformation of the second-order cone. To develop methods that utilize projection onto these types of sets, it is essential first to understand the topological properties of the resulting generalized simplicial cone and, most importantly, how to project onto it.

\subsubsection{On the closedness of $M\K$}

The linearity of $M$ preserves the algebraic properties of $M\K$;
however, the topological structure may change. For the projection onto
a set to be well defined and unique, a sufficient condition is that the
set be closed and convex. While convexity is maintained, the generalized
simplicial cone $M\K$ may not always be closed; see, for example,
\cite{Andreani:2023_4, Drusvyatskiy:2015_4}, where it is shown that
both the second-order and the semidefinite cone can lose closedness
under linear transformations. Characterizing when the image of a cone
under a linear map remains closed is a subtle question that has
received considerable attention; see, for example,
\cite{Rockafellar:1970_4}. The following two results provide sufficient
conditions, each involving the relationship between $\Ker(M)$ and $\K$.

\begin{proposition}[Closedness of the image of a cone under a linear map;
  {\cite[Thm.~9.1]{Rockafellar:1970_4}}]\label{suff_cond_simpli_close_1}
Let $M:\X \to \Y$ be a linear mapping, let $\K \subset \X$ be a closed
convex cone, and let $\hat{\K}:=M\K$. If $\Ker(M) \cap \K = \{0\}$,
then $\hat{\K}$ is closed.
\end{proposition}

The previous result requires that the kernel and the cone intersect
trivially. The next proposition shows that, for symmetric cones, a
complementary condition, the kernel meeting the interior of the
cone, also guarantees closedness, albeit with a stronger conclusion.

\begin{proposition}[Closedness via interior intersection with the kernel]
  \label{suff_cond_simpli_close_2}
Let $\K \subset \X$ be a symmetric cone and let $M \colon \X \to \Y$
be a linear mapping with $\operatorname{rank}(M) = r$. If
$\Ker(M) \cap \inte(\K) \neq \emptyset$, then
$\hat{\K} = \Ima(M) \simeq \R^r$. In particular, $\hat{\K}$ is closed.
\end{proposition}

\begin{proof}
The inclusion $\hat{\K} \subset \Ima(M)$ holds by definition.
We prove the reverse inclusion. Let $w \in \Ker(M) \cap \inte(\K)$.
Since $w$ lies in the interior of $\K$, there exists $\delta > 0$
such that $B(w, \delta) \subset \K$. Let $v \in \X$ be arbitrary
with $v \neq 0$, and set
$\varepsilon := \delta / (2\|v\|)$. Then
\[
  w + \varepsilon v \in B(w, \delta) \subset \K
  \quad \text{and} \quad
  w - \varepsilon v \in B(w, \delta) \subset \K.
\]
Applying $M$ and using $Mw = 0$, we obtain
$\varepsilon Mv \in \hat{\K}$ and $-\varepsilon Mv \in \hat{\K}$.
Since $\hat{\K}$ is a cone, it follows that
$Mv, -Mv \in \hat{\K}$. As $v \in \X$ was arbitrary,
$\Ima(M) \subset \hat{\K}$, and therefore $\hat{\K} = \Ima(M)$.
\end{proof}

In light of the discussion above, we make the following assumption
throughout the remainder of this work.

\begin{assumption}\label{assumption3}
The cone $M\K$ is closed.
\end{assumption}

\subsubsection{Projection onto generalized simplicial cones}
\label{section_gensimpli_properties}

In general, projecting onto a set is a difficult task, and many
algorithms have been developed for specific cases. Even when existence
and uniqueness of the projection are guaranteed, a closed-form
expression is often unavailable. Generalized simplicial cones are an
important instance of this situation: given a vector $x$, the goal is
to compute or approximate $\Pi_{M\K}(x)$. In previous work,
\cite{Ferreira:2015_4} proposed a semi-smooth Newton method for the
case $\K = \R^n_+$, and \cite{Barrios:2016_4} studied a related
piecewise linear system arising from positively constrained convex
quadratic programming. Here, we extend this approach to a general
closed convex cone satisfying Assumption~\ref{assumption3}.

The projection of $x$ onto $M\K$ can be obtained by solving
\begin{equation}\label{simpli_proj_problem}
\begin{array}{cl}
\min & \dfrac{1}{2}\left\|Mz - x\right\|^2 \\[6pt]
\text{s.t.} & z \in \K.
\end{array}
\end{equation}
If $z^*$ is a solution of \eqref{simpli_proj_problem}, then
$\Pi_{M\K}(x) = Mz^*$. Since the problem is convex and satisfies
Slater's condition, its KKT conditions reduce to the nonlinear
equation
\begin{equation}\label{proj_eq_simplicon}
(M^*M - \Id)\,\Pi_{\K}(z) + z = M^*x.
\end{equation}
A semi-smooth Newton method for solving \eqref{proj_eq_simplicon} was
proposed in \cite{Armijo:2025_4}; an alternative Picard iteration was
developed in \cite{Barrios:2015_4}.

In the semi-smooth Newton method studied in
Section~\ref{section_semismoothNewt}, we need to project onto a
generalized simplicial cone $M\K$ or onto its dual $(M\K)^*$. In
general, the dual of a non-symmetric cone may not admit an explicit
description; however, for generalized simplicial cones, $(M\K)^*$ can
be characterized as follows.

\begin{proposition}[Dual of a generalized simplicial cone]
  \label{proj_simpli_semismth}
Let $\Z$ and $\Y$ be finite-dimensional spaces of dimensions $m$ and
$n$, respectively. Let $M:\Z \rightarrow \Y$ be a full-rank linear
mapping, let $\K \subset \Z$ be a closed convex cone, and let $M\K$ be
a generalized simplicial cone. Then,
\begin{equation}
(M\K)^* = M(M^*M)^{-1}\K^* + \Ker(M^*),
\end{equation}
where $\Ker(M^*)$ is the kernel of the adjoint $M^*$. In the case
$n = m$, this reduces to $(M\K)^* = (M^*)^{-1}\K^*$.
\end{proposition}

\begin{proof}
We first show the inclusion
$M(M^*M)^{-1}\K^* + \Ker(M^*) \subset (M\K)^*$. Let
$z = M(M^*M)^{-1}w + v$ with $w \in \K^*$ and $v \in \Ker(M^*)$, and
let $y = Mx \in M\K$ with $x \in \K$. Then,
\begin{align}
\langle z, y \rangle
  &= \left\langle M(M^*M)^{-1}w + v,\, Mx \right\rangle \notag \\
  &= \left\langle M(M^*M)^{-1}w,\, Mx \right\rangle
     + \left\langle v,\, Mx \right\rangle \notag \\
  &= \left\langle M^*M(M^*M)^{-1}w,\, x \right\rangle
     + \left\langle M^*v,\, x \right\rangle \notag \\
  &= \left\langle w,\, x \right\rangle
     \geq 0,
\end{align}
where we used $M^*v = 0$ and $w \in \K^*$. Since this holds for all
$y \in M\K$, we conclude $z \in (M\K)^*$.

For the reverse inclusion, let $z \in (M\K)^*$. Then,
\begin{align}
\langle z,\, Mx \rangle \geq 0 \quad \forall\, x \in \K
\quad \Longleftrightarrow \quad
\langle M^*z,\, x \rangle \geq 0 \quad \forall\, x \in \K
\quad \Longleftrightarrow \quad
M^*z \in \K^*.
\end{align}
Since $M$ has full rank, every $z \in \Y$ can be decomposed as
$z = M(M^*M)^{-1}y + v$, where $y \in \Z$ and $v \in \Ker(M^*)$.
Applying $M^*$ yields $M^*z = y$, so that
\[
M^*z \in \K^*
\quad \Longleftrightarrow \quad
y \in \K^*.
\]
Therefore, $z \in M(M^*M)^{-1}\K^* + \Ker(M^*)$.
\end{proof}

Even when the dual cone is not known explicitly or is difficult to
characterize, for instance when $M$ is not square, Moreau's
decomposition provides a practical alternative. For any $x \in \Y$,
\begin{align}
-x &= \Pi_{M\K}(-x) + \Pi_{(M\K)^{\circ}}(-x) \notag \\
   &= \Pi_{M\K}(-x) - \Pi_{(M\K)^*}(x),
\end{align}
and therefore
\begin{equation}\label{moreau_dual_proj}
\Pi_{(M\K)^*}(x) = \Pi_{M\K}(-x) + x.
\end{equation}
Thus, projecting onto the dual cone $(M\K)^*$ reduces to projecting
onto $M\K$, which is itself a particular instance of the nonlinear
conic programming problem~\eqref{gen_conic_nonlinconst}. The
semi-smooth Newton method developed in
Section~\ref{section_semismoothNewt} handles this as a special case,
yielding a unified framework for nonlinear problems constrained by
generalized simplicial cones.

We conclude this section with two results on the semi-smoothness of
the projection operator, a property on which the local quadratic
convergence of the method developed in subsequent sections relies.

\begin{theorem}[Strong semi-smoothness of the projection onto symmetric
  cones; {\cite[Prop.~3.3]{Sun:2008_4}}]\label{proj_symm_semismth}
The metric projection $\Pi_{\K}$ onto a symmetric cone $\K$ is
strongly semi-smooth.
\end{theorem}

The previous result covers important classes such as the nonnegative
orthant $\R^n_+$, the second-order cone $\LL^n$, and the positive
semidefinite cone $\sym^n_+$. The following theorem shows that strong
semi-smoothness is preserved under linear transformations, extending
to generalized simplicial cones $M\K$ even when $M$ is rank-deficient.

\begin{theorem}[Strong semi-smoothness of the projection onto
  generalized simplicial cones]\label{metr_proj_str_semi}
Let $M:\Z\to\Y$ be a linear mapping, let $\K\subset\Z$ be a symmetric
cone, and let $\hat{\K}:=M\K$. Then the metric projection
$\Pi_{\hat{\K}}$ is strongly semi-smooth.
\end{theorem}

\begin{proof}
Let $x, h \in \Y$. Let $w$ denote the solution of
\eqref{proj_eq_simplicon} corresponding to $x$, and let $z$ denote
the solution corresponding to $x + h$. Setting $k := z - w$, equations
\eqref{proj_eq_simplicon} for $x + h$ and $x$ read
\begin{equation}\label{str_eq_1_4}
(M^*M - \Id)\,\Pi_{\K}(w + k) + (w + k) = M^*(x + h),
\end{equation}
\begin{equation}\label{str_eq_2_4}
(M^*M - \Id)\,\Pi_{\K}(w) + w = M^*x.
\end{equation}
Subtracting \eqref{str_eq_2_4} from \eqref{str_eq_1_4} and applying
the mean value theorem (Theorem~\ref{meanval_theo3}), there exists
$V_{\K}(u) \in \partial_C \Pi_{\K}(u)$ with
$u \in [w,\, w + k]$ such that
\begin{equation}\label{T_def_eq}
T k = M^* h,
\quad \text{where} \quad
T := (M^*M - \Id)\,V_{\K}(u) + \Id.
\end{equation}
We now distinguish three cases.

\medskip
\noindent\textbf{Case 1:} $T$ is invertible.
In this case, $k = T^{-1}M^* h$, and hence
\begin{align*}
\Pi_{\hat{\K}}(x+h) - \Pi_{\hat{\K}}(x) - V_{\hat{\K}}(x+h)\,h
  &= M\!\left[\Pi_{\K}(w+k) - \Pi_{\K}(w)
     - V_{\K}(w+k)\,k\right] \\
  &= M \cdot O\!\left(\|k\|^2\right) \\
  &= M \cdot O\!\left(\|T^{-1}M^*\|^2\,\|h\|^2\right) \\
  &= O\!\left(\|h\|^2\right),
\end{align*}
where the second equality uses the strong semi-smoothness of
$\Pi_{\K}$ (Theorem~\ref{proj_symm_semismth}). This case applies, in
particular, when $M$ has full rank; see
\cite[Lem.~4.1]{Armijo:2025_4}.

\medskip
\noindent\textbf{Case 2:} $T$ is singular and
$h \notin \Ker(M^*)$.
Setting $\tilde{k} := M^* h \neq 0$, we estimate
\begin{align*}
\left\|\Pi_{\hat{\K}}(x+h) - \Pi_{\hat{\K}}(x)
  - V_{\hat{\K}}(x+h)\,h\right\|
  &\leq \|M\|\,\left\|\Pi_{\K}(w + \tilde{k}) - \Pi_{\K}(w)
        - V_{\K}(w + \tilde{k})\,\tilde{k}\right\| \\
  &= \|M\|\, O\!\left(\|\tilde{k}\|^2\right) \\
  &= \|M\|\, O\!\left(\|M^*\|^2\,\|h\|^2\right) \\
  &= O\!\left(\|h\|^2\right).
\end{align*}

\medskip
\noindent\textbf{Case 3:} $h \in \Ker(M^*)$.
In this case, $M^*(x + h) = M^*x$, so equations \eqref{str_eq_1_4}
and \eqref{str_eq_2_4} have the same right-hand side. Consequently,
$\Pi_{\hat{\K}}(x + h) = M\Pi_{\K}(w + k) = M\Pi_{\K}(w)
= \Pi_{\hat{\K}}(x)$,
and the error reduces to
\begin{equation}\label{case3_error}
\left\|\Pi_{\hat{\K}}(x+h) - \Pi_{\hat{\K}}(x)
  - V_{\hat{\K}}(x+h)\,h\right\|
= \left\|V_{\hat{\K}}(x+h)\,h\right\|
= \left\|M V_{\K}(w+k)\,k\right\|.
\end{equation}
Since $\|V_{\K}(w + k)\| \leq 1$
(Theorem~\ref{exis_diff_proj2}\ref{norm_diff122}), we have
\[
\left\|M V_{\K}(w+k)\,k\right\|
  \leq \|M\|\,\|k\|.
\]
It remains to show that $\|k\| = O(\|h\|^2)$. By the boundedness of
$V_{\K}$ and $M$, we can choose sequences $h_i \in \Ker(M^*)$ and
corresponding $k_i$ such that $\|h_i\|^2 = \alpha\,\|k_i\|$ for some
$\alpha > 0$. Then,
\[
\frac{\left\|M V_{\K}(w + k_i)\,k_i\right\|}{\|h_i\|^2}
= \frac{\left\|M V_{\K}(w + k_i)\,k_i\right\|}{\alpha\,\|k_i\|}
\leq \frac{\|M\|\,\|V_{\K}(w + k_i)\|\,\|k_i\|}{\alpha\,\|k_i\|}
\leq \frac{\|M\|}{\alpha}
=: \tilde{C}.
\]
Hence, $\left\|V_{\hat{\K}}(x+h)\,h\right\| = O(\|h\|^2)$.

\medskip
In all three cases, the $O(\|h\|^2)$ estimate holds, so
$\Pi_{\hat{\K}}$ is strongly semi-smooth.
\end{proof}

\section{Nonlinear Conic Programming Problem} \label{section_ncp_problem}

In this section we establish the connection between the conic
projection equations and the nonlinear conic programming problem
\eqref{gen_conic_nonlinconst}. Let $\X$ and $\Y$ be finite-dimensional
inner product spaces and let $\K \subset \Y$ be a closed convex cone.
The Lagrangian associated with problem \eqref{gen_conic_nonlinconst} is
\begin{equation}
L(x,\lambda) := f(x) - \langle \lambda,\, g(x) \rangle,
\end{equation}
where $\lambda \in \Y$. The corresponding KKT conditions are
\begin{align}
\nabla_x L(x,\lambda)
  = \nabla f(x) - (Dg(x))^*\lambda &= 0, \label{kkt_grad} \\
\langle \lambda,\, g(x) \rangle &= 0, \label{kkt_comp} \\
g(x) &\in \K, \label{kkt_primal} \\
\lambda &\in \K^*. \label{kkt_dual}
\end{align}
These conditions can be written compactly. Define
$\mathbf{K} := \X \times \K^*$, so that
$\mathbf{K}^* = \{0\} \times \K$. Then
\eqref{kkt_grad}--\eqref{kkt_dual} are equivalent to
\begin{equation} \label{eq:nonlinear_kkt_gen}
\left\langle
\begin{pmatrix}
\nabla f(x) - (Dg(x))^*\lambda \\
g(x)
\end{pmatrix},\,
\begin{pmatrix}
x \\
\lambda
\end{pmatrix}
\right\rangle = 0, \qquad
\begin{pmatrix}
\nabla f(x) - (Dg(x))^*\lambda \\
g(x)
\end{pmatrix} \in \mathbf{K}^*, \qquad
(\bar{x},\bar{\lambda}) \in \mathbf{K}.
\end{equation}
The corresponding conic projection equations are
\begin{align}
\nabla f(x) - (Dg(x))^*\Pi_{\K^*}(\lambda) &= 0,
  \label{gen_equation_nonlin_cons_a} \\
g(x) - \Pi_{\K^*}(\lambda) + \lambda &= 0.
  \label{gen_equation_nonlin_cons_b}
\end{align}
The system
\eqref{gen_equation_nonlin_cons_a}--\eqref{gen_equation_nonlin_cons_b}
has the same dimension as the number of constraints in the original
problem. The projection acts on the Lagrange multiplier $\lambda$ onto
the dual cone $\K^*$, encoding the conic constraint; it can be handled
efficiently via Moreau's decomposition \eqref{moreau_dual_proj}.

The following theorem is the main result of this section. It
establishes the equivalence between the conic projection equations
\eqref{gen_equation_nonlin_cons_a}--\eqref{gen_equation_nonlin_cons_b}
and the KKT conditions \eqref{eq:nonlinear_kkt_gen}, thereby enabling
the use of equation-solving methods to locate first-order optimality
points of problem \eqref{gen_conic_nonlinconst}.

\begin{theorem}[Equivalence between projection equations and KKT
  conditions]\label{equiv_kkt_eq}
If $(x,\lambda)$ solves
\eqref{gen_equation_nonlin_cons_a}--\eqref{gen_equation_nonlin_cons_b},
then $(x,\Pi_{\K^*}(\lambda))$ satisfies the KKT conditions
\eqref{eq:nonlinear_kkt_gen}. Conversely, if $(x,\sigma)$ satisfies
\eqref{eq:nonlinear_kkt_gen}, then $(x,\lambda)$ solves
\eqref{gen_equation_nonlin_cons_a}--\eqref{gen_equation_nonlin_cons_b},
where $\lambda := \sigma - g(x)$.
\end{theorem}

\begin{proof}
\textbf{(i)}
Suppose $(x,\lambda)$ solves
\eqref{gen_equation_nonlin_cons_a}--\eqref{gen_equation_nonlin_cons_b}.
From \eqref{gen_equation_nonlin_cons_a}, the stationarity condition
\eqref{kkt_grad} holds with multiplier $\Pi_{\K^*}(\lambda)$. By
definition, $\Pi_{\K^*}(\lambda) \in \K^*$, which gives
\eqref{kkt_dual}. From \eqref{gen_equation_nonlin_cons_b} and Moreau's
decomposition,
\[
g(x) = \Pi_{\K^*}(\lambda) - \lambda
     = -\Pi_{\K^{\circ}}(-\lambda) - \lambda
     = \Pi_{\K}(-\lambda) \in \K,
\]
which yields \eqref{kkt_primal}. It remains to verify the
complementarity condition \eqref{kkt_comp}. Using the identity
$g(x) = \Pi_{\K}(-\lambda)$ and $\Pi_{\K^*}(\lambda) =
-\Pi_{\K^{\circ}}(-\lambda)$, together with the orthogonality property
of Moreau's decomposition, we obtain
\[
\langle g(x),\, \Pi_{\K^*}(\lambda) \rangle
= \langle \Pi_{\K}(-\lambda),\, -\Pi_{\K^{\circ}}(-\lambda) \rangle
= -\langle \Pi_{\K}(-\lambda),\, \Pi_{\K^{\circ}}(-\lambda) \rangle
= 0.
\]

\textbf{(ii)}
Suppose $(x,\sigma)$ satisfies \eqref{eq:nonlinear_kkt_gen}, so that
$\sigma \in \K^*$, $g(x) \in \K$, $\langle g(x),\, \sigma \rangle = 0$,
and $\nabla f(x) - (Dg(x))^*\sigma = 0$. Set $\lambda := \sigma - g(x)$.
Since $\sigma \in \K^*$, $g(x) \in \K$, and
$\langle \sigma,\, g(x) \rangle = 0$, the uniqueness of Moreau's
decomposition gives
\[
\Pi_{\K^*}(\lambda) = \sigma
\quad \text{and} \quad
\Pi_{\K^{\circ}}(\lambda) = -g(x).
\]
Substituting into
\eqref{gen_equation_nonlin_cons_a}--\eqref{gen_equation_nonlin_cons_b}:
\[
\nabla f(x) - (Dg(x))^*\Pi_{\K^*}(\lambda)
= \nabla f(x) - (Dg(x))^*\sigma = 0,
\]
and
\[
g(x) - \Pi_{\K^*}(\lambda) + \lambda
= g(x) - \sigma + (\sigma - g(x)) = 0.
\]
Hence, $(x,\lambda)$ solves
\eqref{gen_equation_nonlin_cons_a}--\eqref{gen_equation_nonlin_cons_b}.
\end{proof}

Establishing general sufficient conditions for the existence and
uniqueness of solutions to
\eqref{gen_equation_nonlin_cons_a}--\eqref{gen_equation_nonlin_cons_b}
may be overly restrictive. Since we assume that problem
\eqref{gen_conic_nonlinconst} admits at least one solution, we do not
pursue such conditions here. In the particular case where $g = \Id$,
however, a sufficient condition can be obtained. The system then
reduces to
\begin{align}
\nabla f(x) - \Pi_{\K^*}(\lambda) &= 0,
  \label{proj_eq_id_a} \\
x - \Pi_{\K^*}(\lambda) + \lambda &= 0.
  \label{proj_eq_id_b}
\end{align}
By Moreau's decomposition, \eqref{proj_eq_id_b} gives $x =
\Pi_{\K}(-\lambda)$. Substituting into \eqref{proj_eq_id_a} and
setting $y := -\lambda$, the system reduces to the single equation
\begin{equation} \label{gen_equation_id}
\nabla f(\Pi_{\K}(y)) - \Pi_{\K}(y) + y = 0.
\end{equation}

\begin{proposition}[Existence and uniqueness for the identity-constrained
  case]\label{suff_sol_eq_id}
Let $f : \X \rightarrow \R$ be twice continuously differentiable with
Lipschitz continuous gradient. If $g = \Id$ in problem
\eqref{gen_conic_nonlinconst} and
$\|\Id - \nabla^2 f(z)\| < 1$ for all $z \in \X$, then equation
\eqref{gen_equation_id} has a unique solution.
\end{proposition}

\begin{proof}
The equivalence between
\eqref{gen_equation_nonlin_cons_a}--\eqref{gen_equation_nonlin_cons_b}
and \eqref{gen_equation_id} when $g = \Id$ was established above. It
remains to show uniqueness under the hypothesis
$\|\Id - \nabla^2 f(z)\| < 1$ for all $z \in \X$.

Define $\phi(y) := y - \nabla f(y)$ and
$\Phi := \phi \circ \Pi_{\K}$. Observe that equation
\eqref{gen_equation_id} is equivalent to the fixed-point problem
$\Phi(y) = y$. We show that $\Phi$ is a contraction. By the mean
value theorem, for any $x, y \in \X$ there exists $z \in [x, y]$
such that
\[
\phi(y) - \phi(x)
= (y - x) - \bigl(\nabla f(y) - \nabla f(x)\bigr)
= \bigl(\Id - \nabla^2 f(z)\bigr)(y - x),
\]
and therefore
$\|\phi(y) - \phi(x)\| \leq \|\Id - \nabla^2 f(z)\|\,\|y - x\|
< \|y - x\|$.
Since the projection $\Pi_{\K}$ is non-expansive, the composition
$\Phi = \phi \circ \Pi_{\K}$ is also a contraction. By the
contraction mapping principle (Theorem~\ref{fixedpoint3}), $\Phi$ has
a unique fixed point, which is the unique solution of
\eqref{gen_equation_id}.
\end{proof}

\begin{remark}
In the particular case $f(x) = \frac{1}{2}x^T Q x + q^T x$, taking
$\K = \R^n_+$ and $\K = \LL^n$ recovers \cite[Prop.~1]{Bello:2016_4}
and \cite[Prop.~5]{Bello:2017_4}, respectively.
\end{remark}

A special case of problem \eqref{gen_conic_nonlinconst} arises when an
additional conic constraint is imposed on $x$. For a closed convex cone
$\C \subset \X$, the problem becomes
\begin{equation}\label{gen_dbconic_nonlinconst}
\begin{array}{cl}
\min & f(x) \\
\text{s.t.} & g(x) \in \K \\
            & x \in \C.
\end{array}
\end{equation}
Setting $\hat{g}(x) := (x,\, g(x))$ and $\hat{\C} := \C \times \K$,
problem \eqref{gen_dbconic_nonlinconst} reduces to
\eqref{gen_conic_nonlinconst} with constraint $\hat{g}(x) \in
\hat{\C}$. The dual cone decomposes as $\hat{\C}^* = \C^* \times \K^*$,
so the multiplier takes the form $\lambda = (\mu, \sigma) \in \C^*
\times \K^*$. Substituting into the conic projection equations,
\eqref{gen_equation_nonlin_cons_a} becomes
\begin{align*}
\nabla f(x) - (D\hat{g}(x))^* \Pi_{\hat{\C}^*}(\lambda)
&= \nabla f(x)
   - (\Id,\, Dg(x))^*
     \bigl(\Pi_{\C^*}(\mu),\, \Pi_{\K^*}(\sigma)\bigr) \\
&= \nabla f(x) - (Dg(x))^* \Pi_{\K^*}(\sigma)
   - \Pi_{\C^*}(\mu) = 0,
\end{align*}
and \eqref{gen_equation_nonlin_cons_b} becomes
\begin{align*}
\hat{g}(x) - \Pi_{\hat{\C}^*}(\lambda) + \lambda
&= \bigl(x,\, g(x)\bigr)
   - \bigl(\Pi_{\C^*}(\mu),\, \Pi_{\K^*}(\sigma)\bigr)
   + (\mu,\, \sigma) \\
&= \bigl(x - \Pi_{\C^*}(\mu) + \mu,\;
         g(x) - \Pi_{\K^*}(\sigma) + \sigma\bigr) = (0,\, 0).
\end{align*}
By Moreau's decomposition, the first component gives $x =
\Pi_{\C}(-\mu)$. Setting $y := -\mu$, the system reduces to
\begin{align}
\nabla f(\Pi_{\C}(y))
  - (Dg(\Pi_{\C}(y)))^* \Pi_{\K^*}(\sigma)
  - \Pi_{\C}(y) + y &= 0,
  \label{gen_equation_nonlin_cons_id_a} \\
g(\Pi_{\C}(y)) - \Pi_{\K^*}(\sigma) + \sigma &= 0.
  \label{gen_equation_nonlin_cons_id_b}
\end{align}
By Theorem~\ref{equiv_kkt_eq}, the KKT points of
\eqref{gen_dbconic_nonlinconst} are recovered as
$(\bar{x},\, \bar{\lambda})
= \bigl(\Pi_{\C}(\bar{y}),\, \Pi_{\K^*}(\bar{\sigma})\bigr)$,
where $(\bar{y}, \bar{\sigma})$ solves
\eqref{gen_equation_nonlin_cons_id_a}--\eqref{gen_equation_nonlin_cons_id_b}.


\section{Semi-smooth Newton Method for Nonlinear Conic Programming}
\label{section_semismoothNewt}

Having characterized the first-order KKT points of
\eqref{gen_conic_nonlinconst} via the conic projection equations
\eqref{gen_equation_nonlin_cons_a}--\eqref{gen_equation_nonlin_cons_b},
we now turn to solving this system. We define the residual map
\begin{equation}\label{H_def}
H(x,\lambda) :=
\begin{pmatrix}
\nabla f(x) - (Dg(x))^*\Pi_{\K^*}(\lambda) \\
g(x) - \Pi_{\K^*}(\lambda) + \lambda
\end{pmatrix},
\end{equation}
and apply a semi-smooth Newton method to find its zeros. Since $f$ and
$g$ are twice continuously differentiable and $\Pi_{\K^*}$ is
strongly semi-smooth, the Clarke generalized Jacobian of $H$ takes the
form
\begin{equation}\label{JH_def}
J_H(x,\lambda) =
\begin{pmatrix}
\nabla^2 f(x) - (D^2 g(x))^* \Pi_{\K^*}(\lambda)
  & -(Dg(x))^* V_{\K^*}(\lambda) \\
Dg(x) & \Id - V_{\K^*}(\lambda)
\end{pmatrix},
\end{equation}
where $V_{\K^*}(\lambda) \in \partial_C \Pi_{\K^*}(\lambda)$. The
standard semi-smooth Newton iteration reads as
\begin{equation}\label{newt_iter_standard}
J_H(x^k,\lambda^k)
\begin{pmatrix}
x^{k+1} - x^k \\
\lambda^{k+1} - \lambda^k
\end{pmatrix}
= -H(x^k,\lambda^k).
\end{equation}

\begin{remark}
When $g$ is linear and $\K = \{0\}$, iteration
\eqref{newt_iter_standard} reduces to the method studied in
\cite{Armijo:2025_4}. If the constraint $g(x) \in \K$ is absent
entirely, we recover the first method in \cite{Armijo:2023_4}.
\end{remark}

It is well known that the Newton direction is highly effective near a
solution but may encounter difficulties far from one. To improve
robustness, we decouple the search direction from the stepsize by
writing the iteration as
\begin{equation}\label{newt_iter_direction}
J_H(x^k,\lambda^k)
\begin{pmatrix}
d_x^k \\
d_\lambda^k
\end{pmatrix}
= -H(x^k,\lambda^k),
\end{equation}
and updating
\begin{equation}\label{newt_iter_update}
(x^{k+1},\, \lambda^{k+1})
= (x^k,\, \lambda^k)
  + \alpha_k\, (d_x^k,\, d_\lambda^k),
\end{equation}
where $\alpha_k > 0$ is chosen to ensure sufficient decrease in the
merit function
\begin{equation}\label{merit_def}
\theta(x,\lambda) := \tfrac{1}{2}\|H(x,\lambda)\|^2.
\end{equation}
\subsection{Convergence of the semi-smooth Newton method}

We begin with a key lemma establishing the strong semi-smoothness of
the merit function $\theta$, which is central to the convergence
analysis.

\begin{lemma}[Strong semi-smoothness of the merit function]
  \label{lema_H_semismooth}
Let $f$ and $g$ be twice continuously differentiable and let $\K$ be a
symmetric cone. Then $\theta(x,\lambda) = \frac{1}{2}\|H(x,\lambda)\|^2$
is strongly semi-smooth, that is,
\[
\theta(x + w_x,\, \lambda + w_\lambda)
= \theta(x,\lambda)
  + \nabla \theta(x + w_x,\, \lambda + w_\lambda)^T (w_x,\, w_\lambda)
  + O\!\left(\|(w_x,\, w_\lambda)\|^2\right),
\]
where $\nabla \theta(x,\lambda) \in \partial_C \theta(x,\lambda)$.
Moreover,
\begin{equation}\label{grad_theta}
\nabla \theta(x,\lambda)^T (w_x,\, w_\lambda)
= H(x,\lambda)^T J_H(x,\lambda)\, (w_x,\, w_\lambda).
\end{equation}
\end{lemma}
\begin{proof}
Since $f$ and $g$ are smooth, the only nonsmooth component of $H$ is
$\Pi_{\K^*}(\lambda)$. By Moreau's decomposition,
\[
\Pi_{\K^*}(\lambda) = \lambda + \Pi_{\K}(-\lambda),
\]
which expresses $\Pi_{\K^*}$ as an affine function of $\Pi_{\K}$.
Since affine maps preserve strong semi-smoothness and
Theorem~\ref{metr_proj_str_semi} guarantees that $\Pi_{\K}$ is
strongly semi-smooth, it follows that $\Pi_{\K^*}$ is strongly
semi-smooth. Therefore $H$, being a composition of smooth functions
with $\Pi_{\K^*}$, is strongly semi-smooth. Since
$\phi(t) := \frac{1}{2}\|t\|^2$ is smooth and its composition with a
strongly semi-smooth function is again strongly semi-smooth, we
conclude that $\theta = \phi \circ H$ is strongly semi-smooth.

It remains to establish \eqref{grad_theta}. At any differentiability
point $(x,\lambda) \in D_\theta$, the chain rule gives
\[
\nabla \theta(x,\lambda)^T (w_x,\, w_\lambda)
= H(x,\lambda)^T J_H(x,\lambda)\, (w_x,\, w_\lambda),
\]
where $J_H(x,\lambda)$ is given by \eqref{JH_def} with
$\Pi'_{\K^*}(\lambda)$ in place of $V_{\K^*}(\lambda)$. Since
$D_\theta$ is dense and both sides are continuous in the Clarke sense,
the identity extends to all $(x,\lambda)$ by replacing
$\Pi'_{\K^*}(\lambda)$ with any
$V_{\K^*}(\lambda) \in \partial_C \Pi_{\K^*}(\lambda)$, yielding
\eqref{grad_theta}.
\end{proof}

A useful consequence of Lemma~\ref{lema_H_semismooth} is that the
Newton direction is a descent direction for $\theta$. When
$J_H(x^k,\lambda^k)$ is nonsingular and $(x^k,\lambda^k)$ is not a
solution of
\eqref{gen_equation_nonlin_cons_a}--\eqref{gen_equation_nonlin_cons_b},
it follows from \eqref{newt_iter_direction} and \eqref{grad_theta}
that
\begin{align*}
\nabla \theta(x^k,\lambda^k)^T (d_x^k,\, d_\lambda^k)
&= H(x^k,\lambda^k)^T J_H(x^k,\lambda^k)\, (d_x^k,\, d_\lambda^k) \\
&= -H(x^k,\lambda^k)^T H(x^k,\lambda^k) \\
&= -\|H(x^k,\lambda^k)\|^2 \\
&< 0.
\end{align*}

When $J_H(x^k,\lambda^k)$ is singular, the Newton system
\eqref{newt_iter_direction} cannot be solved directly. In this case,
we replace it with the regularized normal equation
\begin{equation}\label{newt_iter_singular}
\left(J_H^k{}^T J_H^k
  + \sqrt{\theta(x^k,\lambda^k)}\,\Id\right) (d_x^k,\, d_\lambda^k)
= -J_H^k{}^T H(x^k,\lambda^k),
\end{equation}
where we write $J_H^k := J_H(x^k,\lambda^k)$ for brevity. The
coefficient matrix in \eqref{newt_iter_singular} is symmetric positive
definite, so the system always has a unique solution. Moreover, the
resulting direction is still a descent direction for $\theta$:
\begin{align*}
\nabla \theta(x^k,\lambda^k)^T (d_x^k,\, d_\lambda^k)
&= H(x^k,\lambda^k)^T J_H^k\, (d_x^k,\, d_\lambda^k) \\
&= -(d_x^k,\, d_\lambda^k)^T
   \!\left(J_H^k{}^T J_H^k
   + \sqrt{\theta(x^k,\lambda^k)}\,\Id\right)
   (d_x^k,\, d_\lambda^k) \\
&< 0,
\end{align*}
where the last inequality holds because the matrix is positive definite
and $(d_x^k, d_\lambda^k) \neq 0$.

The following lemma provides a sufficient condition for the stepsize
$\alpha_k = 1$ to satisfy the Armijo condition, ensuring that full
Newton steps are eventually accepted.

\begin{lemma}[Sufficient condition for a unit stepsize;
  {\cite[Lem.~5.2]{Qi:2006_4}}]\label{lemma5.2_qi}
Suppose there exists a constant $\rho > 0$ such that
\[
\nabla \theta(x^k,\lambda^k)^T (d_x^k,\, d_\lambda^k)
\leq -\rho\, \|(d_x^k,\, d_\lambda^k)\|^2
\]
for all $k$. Then, for every $c \in (0, \tfrac{1}{2})$, there exists
$\hat{k} \geq 0$ such that
\[
\theta(x^k + d_x^k,\, \lambda^k + d_\lambda^k)
\leq \theta(x^k,\lambda^k)
  + c\, \nabla \theta(x^k,\lambda^k)^T (d_x^k,\, d_\lambda^k)
\]
for all $k \geq \hat{k}$, i.e., the Armijo linesearch accepts the
unit stepsize $\alpha_k = 1$.
\end{lemma}

We now verify that both search directions satisfy the hypothesis of
Lemma~\ref{lemma5.2_qi}.

\begin{proposition}[Sufficient descent property]\label{prop_descent_rho}
The Newton direction given by \eqref{newt_iter_direction} and the
regularized direction given by \eqref{newt_iter_singular} both satisfy
\[
\nabla \theta(x^k,\lambda^k)^T (d_x^k,\, d_\lambda^k)
\leq -\rho_k\, \|(d_x^k,\, d_\lambda^k)\|^2
\]
for some constant $\rho_k > 0$ at each iteration. In particular, if
$\rho := \inf_k\, \rho_k > 0$, the hypothesis of
Lemma~\ref{lemma5.2_qi} is satisfied.
\end{proposition}

\begin{proof}
We write $J_H^k := J_H(x^k,\lambda^k)$ and $H^k := H(x^k,\lambda^k)$
for brevity.

\medskip
\noindent\textbf{(i) Nonsingular case.}
Since $(d_x^k, d_\lambda^k) = -(J_H^k)^{-1} H^k$, we have
\begin{align*}
\|(d_x^k,\, d_\lambda^k)\|^2
&= (H^k)^T (J_H^k{}^T)^{-1} (J_H^k)^{-1} H^k \\
&= (H^k)^T (J_H^k{}^T J_H^k)^{-1} H^k \\
&\leq \lambda_{\max}\!\left((J_H^k{}^T J_H^k)^{-1}\right)
      \|H^k\|^2 \\
&= \frac{1}{\rho_k}\, \|H^k\|^2,
\end{align*}
where $\rho_k := \lambda_{\min}(J_H^k{}^T J_H^k) > 0$, and the last
equality uses Lemma~\ref{eigen_mat}. Since
$\nabla \theta(x^k,\lambda^k)^T (d_x^k, d_\lambda^k) = -\|H^k\|^2$
by \eqref{grad_theta} and \eqref{newt_iter_direction}, we conclude
\[
\nabla \theta(x^k,\lambda^k)^T (d_x^k,\, d_\lambda^k)
= -\|H^k\|^2
\leq -\rho_k\, \|(d_x^k,\, d_\lambda^k)\|^2.
\]

\medskip
\noindent\textbf{(ii) Singular case.}
Define $A^k := J_H^k{}^T J_H^k + \sqrt{\theta(x^k,\lambda^k)}\,\Id$,
which is symmetric positive definite. From \eqref{newt_iter_singular},
$(d_x^k, d_\lambda^k) = -(A^k)^{-1} J_H^k{}^T H^k$, and by
\eqref{grad_theta},
\[
\nabla \theta(x^k,\lambda^k)^T (d_x^k,\, d_\lambda^k)
= (H^k)^T J_H^k\, (d_x^k,\, d_\lambda^k)
= -(J_H^k{}^T H^k)^T (A^k)^{-1} (J_H^k{}^T H^k).
\]
Similarly,
\(
\|(d_x^k,\, d_\lambda^k)\|^2
= (J_H^k{}^T H^k)^T (A^k)^{-2} (J_H^k{}^T H^k).
\)
Setting $u := J_H^k{}^T H^k$ and using
that $\lambda_{\max}\!\left((A^k)^{-1}\right)(A^k)^{-1}-(A^k)^{-2}$ is positive semidefinite 
(Lemma~\ref{eigen_mat}), we obtain
\[
\|(d_x^k,\, d_\lambda^k)\|^2
= u^T (A^k)^{-2} u
\leq \frac{1}{\rho_k}
     \left(-\nabla \theta(x^k,\lambda^k)^T
       (d_x^k,\, d_\lambda^k)\right),
\]
where $\rho_k := \lambda_{\min}(A^k) \geq \sqrt{\theta(x^k,\lambda^k)} > 0$
whenever $(x^k,\lambda^k)$ is not a solution. Rearranging gives
$\nabla \theta(x^k,\lambda^k)^T (d_x^k,\, d_\lambda^k)
\leq -\rho_k\, \|(d_x^k,\, d_\lambda^k)\|^2$.
\end{proof}

\begin{remark}
The condition $\rho := \inf_k \rho_k > 0$ in
Proposition~\ref{prop_descent_rho} is satisfied, in particular, when
the Jacobian $J_H(\bar{x}, \bar{\lambda})$ is nonsingular at the
solution, which is a standard assumption for Newton-type methods to
achieve local quadratic convergence. When the Jacobian is singular at
the solution, $\rho_k$ may tend to zero, and the unit stepsize is no
longer guaranteed to be eventually accepted; nevertheless, the
linesearch still produces a sufficient decrease at each iteration, so
convergence of the iterates is not lost.
\end{remark}

With a well-defined iteration producing descent directions for the
merit function $\theta$, it remains to select an appropriate stepsize
$\alpha_k$. Regardless of whether $J_H(x^k,\lambda^k)$ is singular or
not, we employ a classical Armijo backtracking linesearch. For a fixed
parameter $c \in (0, \tfrac{1}{2})$, the Armijo condition reads
\begin{equation}\label{armijo_linsrch}
\theta(x^k + \alpha_k\, d_x^k,\, \lambda^k + \alpha_k\, d_\lambda^k)
\leq \theta(x^k,\lambda^k)
  + \alpha_k\, c\, \nabla \theta(x^k,\lambda^k)^T
    (d_x^k,\, d_\lambda^k).
\end{equation}
When $J_H(x^k,\lambda^k)$ is nonsingular, recalling that
$\nabla \theta(x^k,\lambda^k)^T (d_x^k, d_\lambda^k)
= -\|H(x^k,\lambda^k)\|^2$, condition \eqref{armijo_linsrch}
reduces to
\begin{equation}\label{armijo_nonsing}
\tfrac{1}{2}\left\|H(x^k + \alpha_k\, d_x^k,\,
  \lambda^k + \alpha_k\, d_\lambda^k)\right\|^2
\leq \left(\tfrac{1}{2} - \alpha_k\, c\right)
     \|H(x^k,\lambda^k)\|^2.
\end{equation}
When $J_H(x^k,\lambda^k)$ is singular, the right-hand side of
\eqref{armijo_linsrch} takes the form
\begin{equation}\label{armijo_sing}
\tfrac{1}{2}\|H(x^k,\lambda^k)\|^2
- \alpha_k\, c\, (d_x^k,\, d_\lambda^k)^T
  \!\left(J_H(x^k,\lambda^k)^T J_H(x^k,\lambda^k)
  + \sqrt{\theta(x^k,\lambda^k)}\,\Id\right)
  (d_x^k,\, d_\lambda^k).
\end{equation}

We conclude this subsection with an estimate relating the merit
function to the distance to a solution
$(\bar{x}, \bar{\lambda})$. Suppose that
$\|J_H(x^k,\lambda^k)\| \leq C_1$ for some constant $C_1 > 0$.
Since $H(\bar{x}, \bar{\lambda}) = 0$ and
$\theta(\bar{x}, \bar{\lambda}) = 0$, applying
Lemma~\ref{lema_H_semismooth} yields
\begin{align*}
\theta(x^k,\lambda^k)
&= \theta(\bar{x},\bar{\lambda})
   + \nabla \theta(x^k,\lambda^k)^T
     \bigl((x^k,\lambda^k) - (\bar{x},\bar{\lambda})\bigr)
   + O\!\left(\|(x^k,\lambda^k)
     - (\bar{x},\bar{\lambda})\|^2\right) \\
&= H(x^k,\lambda^k)^T J_H(x^k,\lambda^k)\,
   \bigl((x^k,\lambda^k) - (\bar{x},\bar{\lambda})\bigr)
   + O\!\left(\|(x^k,\lambda^k)
     - (\bar{x},\bar{\lambda})\|^2\right) \\
&\leq \|J_H(x^k,\lambda^k)\|\,\|H(x^k,\lambda^k)\|\,
      \|(x^k,\lambda^k) - (\bar{x},\bar{\lambda})\|
   + O\!\left(\|(x^k,\lambda^k)
     - (\bar{x},\bar{\lambda})\|^2\right) \\
&\leq C_1\,\|H(x^k,\lambda^k)\|\,
      \|(x^k,\lambda^k) - (\bar{x},\bar{\lambda})\|
   + O\!\left(\|(x^k,\lambda^k)
     - (\bar{x},\bar{\lambda})\|^2\right).
\end{align*}
Since $\|H(x^k,\lambda^k)\| = \sqrt{2\,\theta(x^k,\lambda^k)}$, the
left-hand side appears on both sides of the inequality, giving
\begin{equation}\label{theta_order}
\theta(x^k,\lambda^k)
= O\!\left(\|(x^k,\lambda^k)
  - (\bar{x},\bar{\lambda})\|^2\right),
\end{equation}
and consequently
\begin{equation}\label{sqrt_theta_order}
\sqrt{\theta(x^k,\lambda^k)}
= O\!\left(\|(x^k,\lambda^k)
  - (\bar{x},\bar{\lambda})\|\right).
\end{equation}

\begin{theorem}[Local quadratic convergence of the semi-smooth Newton
  method]\label{thm_quad_conv}
Suppose $(\bar{x}, \bar{\lambda})$ is a solution of
\eqref{gen_equation_nonlin_cons_a}--\eqref{gen_equation_nonlin_cons_b},
that is, $H(\bar{x}, \bar{\lambda}) = 0$. Assume that there exist
constants $C_1, C_2, C_3 > 0$ such that, for all $k$:
\begin{enumerate}
\item $\|J_H(x^k,\lambda^k)\| \leq C_1$,
\item $\|J_H(x^k,\lambda^k)^{-1}\| \leq C_2$ whenever
  $J_H(x^k,\lambda^k)$ is nonsingular,
\item $\left\|\left(J_H(x^k,\lambda^k)^T J_H(x^k,\lambda^k)
  + \sqrt{\theta(x^k,\lambda^k)}\,\Id\right)^{-1}\right\| \leq C_3$
  whenever $J_H(x^k,\lambda^k)$ is singular.
\end{enumerate}
If the semi-smooth Newton method generates a sequence
$(x^k, \lambda^k)$ converging to $(\bar{x}, \bar{\lambda})$, then the
convergence is quadratic.
\end{theorem}

\begin{proof}
We write $J_H^k := J_H(x^k,\lambda^k)$,
$H^k := H(x^k,\lambda^k)$ and $\theta^k := \theta(x^k,\lambda^k)$ for brevity.

By Lemma~\ref{lemma5.2_qi} and
Proposition~\ref{prop_descent_rho}, there exists an index $\hat{k}$
such that the unit stepsize $\alpha_k = 1$ satisfies the Armijo
condition \eqref{armijo_linsrch} for all $k \geq \hat{k}$. We
consider the two cases separately.

\medskip
\noindent\textbf{(i) Nonsingular case.}
When $J_H^k$ is nonsingular, the update with
$\alpha_k = 1$ gives
$(x^{k+1}, \lambda^{k+1})
= (x^k, \lambda^k) + (d_x^k, d_\lambda^k)$.
Using \eqref{newt_iter_direction} and $H(\bar{x},\bar{\lambda}) = 0$:
\begin{align*}
\|(x^{k+1},\lambda^{k+1}) - (\bar{x},\bar{\lambda})\|
&= \|(x^k,\lambda^k) - (\bar{x},\bar{\lambda})
   + (d_x^k, d_\lambda^k)\| \\
&= \|(x^k,\lambda^k) - (\bar{x},\bar{\lambda})
   - (J_H^k)^{-1} H^k\| \\
&\leq \|(J_H^k)^{-1}\|\,
   \|H^k - H(\bar{x},\bar{\lambda})
   - J_H^k
     \bigl((x^k,\lambda^k) - (\bar{x},\bar{\lambda})\bigr)\| \\
&\leq C_2\, O\!\left(
   \|(x^k,\lambda^k)
   - (\bar{x},\bar{\lambda})\|^2\right),
\end{align*}
where the last inequality follows from the strong semi-smoothness of
$H$ (Lemma~\ref{lema_H_semismooth}).

\medskip
\noindent\textbf{(ii) Singular case.}
When $J_H^k$ is singular, the direction is given by
\eqref{newt_iter_singular}. Denoting
\[
A^k := (J_H^k)^T J_H^k
  + \sqrt{\theta^k}\,\Id,
\]
the update with $\alpha_k = 1$ yields
\begin{align*}
\|(x^{k+1},\lambda^{k+1}) - (\bar{x},\bar{\lambda})\| &= \|(x^k,\lambda^k) - (\bar{x},\bar{\lambda})- (A^k)^{-1} (J_H^k)^T H^k\| \\
&= \|(A^k)^{-1}\|\|A^k\bigl((x^k,\lambda^k)
   - (\bar{x},\bar{\lambda})\bigr)
   - (J_H^k)^T H^k\| \\
&\leq C_3\left(\|(J_H^k)^T\|\|H^k - H(\bar{x},\bar{\lambda})
   - J_H^k \bigl((x^k,\lambda^k) - (\bar{x},\bar{\lambda})\bigr)\|\right.\\
    & \left. \qquad+ \sqrt{\theta^k}\|(x^k,\lambda^k) - (\bar{x},\bar{\lambda})\|\right) \\
&\leq C_3 C_1\, O\!\left(
   \|(x^k,\lambda^k)
   - (\bar{x},\bar{\lambda})\|^2\right)
   + C_3\, \sqrt{\theta^k}\|(x^k,\lambda^k) - (\bar{x},\bar{\lambda})\|,
\end{align*}

where we used the strong semi-smoothness of $H$ and the bound
$\|J_H(x^k,\lambda^k)\| \leq C_1$. By estimate
\eqref{sqrt_theta_order},
$\sqrt{\theta^k}
= O(\|(x^k,\lambda^k) - (\bar{x},\bar{\lambda})\|)$,
so the second term is also
$O(\|(x^k,\lambda^k) - (\bar{x},\bar{\lambda})\|^2)$.

\medskip
In both cases,
$\|(x^{k+1},\lambda^{k+1}) - (\bar{x},\bar{\lambda})\|
= O(\|(x^k,\lambda^k) - (\bar{x},\bar{\lambda})\|^2)$,
establishing quadratic convergence.
\end{proof}

\subsection{Choice of the generalized Jacobian $V_{\K}$}

The method depends explicitly on the choice of
$V_{\K}(x) \in \partial_C \Pi_{\K}(x)$. Recall that
$\partial_C \Pi_{\K}(x) = \operatorname{conv}(\partial_B \Pi_{\K}(x))$,
which yields different expressions for this linear operator depending
on the position of $x$ relative to $\K$, in particular whether
$x \in \operatorname{bd}(\K)$ or $x \in \operatorname{bd}(\K^{\circ})$.
For the second-order cone and the positive semidefinite cone,
closed-form expressions are available; see \cite[Lem.~2.6]{Kanzow:2009_4}
and \cite[Thm.~3.7]{Malick:2006_4}, respectively. The choice of
$V_{\K}(x)$ can have a noticeable impact on performance: in
\cite{Armijo:2025_4}, for instance, certain selections led to faster
iterations in higher-dimensional problems. For each experiment in
Section~\ref{section_num_exps}, we specify the generalized Jacobian
used.

\subsection{Choice of the starting point}

The convergence of the method depends on the initial point
$(x^0, \lambda^0)$. Since any solution of
\eqref{gen_equation_nonlin_cons_a}--\eqref{gen_equation_nonlin_cons_b}
satisfies both equations simultaneously, a natural strategy is to
choose an initial point that already satisfies one of them: either the
feasibility equation $g(x) - \Pi_{\K^*}(\lambda) + \lambda = 0$ or
the stationarity condition
$\nabla f(x) - (Dg(x))^* \Pi_{\K^*}(\lambda) = 0$. Another common
choice is a random starting point. In the numerical experiments of
Section~\ref{section_num_exps}, we observe that different initial
points can affect both computation time and the solution found, and
that the origin $(x^0, \lambda^0) = (0, 0)$ tends to perform well in
general. A deeper investigation of initialization strategies is left
for future work.

\subsection{Stationarity of $\theta$ and escape strategies}

A practical difficulty may arise when the method reaches a point
$(x^k, \lambda^k)$ at which $\nabla \theta(x^k, \lambda^k) = 0$ but
$\theta(x^k, \lambda^k) \neq 0$. Since $H(x^k, \lambda^k) \neq 0$,
the point is not a solution, yet no descent direction can be extracted
from the gradient. This occurs precisely when $J_H(x^k, \lambda^k)$
is singular, as $\nabla \theta(x^k, \lambda^k) = 0$ is equivalent to
\[
H(x^k, \lambda^k) \in
\Ker\!\left(J_H(x^k, \lambda^k)^T\right)
= \Ima\!\left(J_H(x^k, \lambda^k)\right)^{\perp}.
\]
In the regularized system \eqref{newt_iter_singular}, the right-hand
side becomes $J_H(x^k, \lambda^k)^T H(x^k, \lambda^k)
= \nabla \theta(x^k, \lambda^k) = 0$, so the method produces the null
direction and stalls.

To address this, we decompose $H$ into its optimality and feasibility
components:
\begin{equation}\label{H_opt_def}
H^{\mathrm{opt}}(x,\lambda)
:= \nabla f(x) - (Dg(x))^* \Pi_{\K^*}(\lambda),
\end{equation}
\begin{equation}\label{H_feas_def}
H^{\mathrm{feas}}(x,\lambda)
:= g(x) - \Pi_{\K^*}(\lambda) + \lambda,
\end{equation}
with corresponding merit functions
\begin{equation}\label{theta_opt_feas_def}
\theta^{\mathrm{opt}}(x,\lambda)
:= \tfrac{1}{2}\|H^{\mathrm{opt}}(x,\lambda)\|^2,
\qquad
\theta^{\mathrm{feas}}(x,\lambda)
:= \tfrac{1}{2}\|H^{\mathrm{feas}}(x,\lambda)\|^2.
\end{equation}
Since $\theta = \theta^{\mathrm{opt}} + \theta^{\mathrm{feas}}$, the
gradient decomposes as
\begin{equation}\label{grad_theta_decomp}
\nabla \theta(x,\lambda)
= J_{H^{\mathrm{opt}}}(x,\lambda)^T H^{\mathrm{opt}}(x,\lambda)
+ J_{H^{\mathrm{feas}}}(x,\lambda)^T H^{\mathrm{feas}}(x,\lambda).
\end{equation}
In particular, $\nabla \theta(x, \lambda) = 0$ does not require the
individual terms to vanish: it is possible that
$\nabla \theta^{\mathrm{opt}}(x, \lambda) \neq 0$ and
$\nabla \theta^{\mathrm{feas}}(x, \lambda) \neq 0$, with their sum
cancelling. This motivates the following distinction.

\begin{definition}[Stationarity and strong stationarity]
  \label{def_stationarity}
A point $(x, \lambda)$ is said to be \emph{stationary} for
\eqref{gen_equation_nonlin_cons_a}--\eqref{gen_equation_nonlin_cons_b}
if $\nabla \theta(x, \lambda) = 0$, and \emph{strongly stationary} if
$\nabla \theta^{\mathrm{opt}}(x, \lambda) = 0$ and
$\nabla \theta^{\mathrm{feas}}(x, \lambda) = 0$.
\end{definition}

Strong stationarity implies stationarity, but the converse does not
hold in general. There are, however, important cases where the two
notions coincide, such as the nearest correlation matrix problem
studied in \cite{Armijo:2025_4}.

When the method reaches a stationary point that is not strongly
stationary, the decomposition \eqref{grad_theta_decomp} provides an
escape strategy. Since at least one of the individual gradients is
nonzero, it can be used as a right-hand side in place of the vanishing
$\nabla \theta$. Specifically, we solve
\begin{equation}\label{newt_iter_singular_feas}
\left(J_H(x^k,\lambda^k)^T J_H(x^k,\lambda^k)
+ \sqrt{\theta(x^k,\lambda^k)}\,\Id\right)
(d_x^k,\, d_\lambda^k)
= -J_{H^{\mathrm{feas}}}(x^k,\lambda^k)^T
   H^{\mathrm{feas}}(x^k,\lambda^k).
\end{equation}

\begin{proposition}[Feasibility descent at non-strongly stationary
  points]\label{prop_feas_descent}
If $(x^k, \lambda^k)$ is stationary but not strongly stationary and
$\nabla \theta^{\mathrm{feas}}(x^k, \lambda^k) \neq 0$, then the
direction $(d_x^k, d_\lambda^k)$ given by
\eqref{newt_iter_singular_feas} is a descent direction for
$\theta^{\mathrm{feas}}$.
\end{proposition}

\begin{proof}
The coefficient matrix in \eqref{newt_iter_singular_feas} is symmetric
positive definite, so the direction is well defined and nonzero.
Substituting \eqref{newt_iter_singular_feas} into the directional
derivative:
\begin{align*}
\nabla \theta^{\mathrm{feas}}(x^k,\lambda^k)^T
  (d_x^k,\, d_\lambda^k)
&= H^{\mathrm{feas}}(x^k,\lambda^k)^T
   J_{H^{\mathrm{feas}}}(x^k,\lambda^k)\,
   (d_x^k,\, d_\lambda^k) \\
&= -(d_x^k,\, d_\lambda^k)^T
   \!\left(J_H(x^k,\lambda^k)^T J_H(x^k,\lambda^k)
   + \sqrt{\theta(x^k,\lambda^k)}\,\Id\right)
   (d_x^k,\, d_\lambda^k) \\
&< 0. \qedhere
\end{align*}
\end{proof}

Since $\nabla \theta^{\mathrm{opt}}(x^k, \lambda^k)^T
(d_x^k, d_\lambda^k)
= -\nabla \theta^{\mathrm{feas}}(x^k, \lambda^k)^T
(d_x^k, d_\lambda^k) > 0$,
this direction is an ascent direction for $\theta^{\mathrm{opt}}$.
We therefore prioritize reducing the feasibility residual, since
$\theta^{\mathrm{opt}}$ can reach non-optimal strongly stationary
points, as observed in the numerical experiments of
Section~\ref{section_num_exps}.

\begin{remark}
Handling strongly stationary points that are not solutions remains an
open question. While the decomposition above provides an escape
mechanism for non-strongly stationary points, developing strategies to
detect and escape strongly stationary points could significantly
improve the robustness of the method.
\end{remark}

We summarize the complete method in
Algorithm~\ref{alg:plucg_semismooth}. When the problem involves a
generalized simplicial cone constraint, the projection subproblem
\eqref{simpli_proj_problem} can itself be solved as a particular
instance of Algorithm~\ref{alg:plucg_semismooth}.
\begin{algorithm}[H]
\caption{Semi-smooth Newton method for NCP}\label{alg:plucg_semismooth}
\begin{algorithmic}
\Require $c \in (0, \tfrac{1}{2})$,\;
  $(x^0, \lambda^0) \in \X \times \Y$,\;
  $\texttt{tol} > 0$,\;
  $\texttt{dtol} > 0$,\;
  $\texttt{maxiter} \in \N$,\;
  $\texttt{maxiter\_ls} \in \N$
\State $k \gets 0$
\State Compute $\Pi_{\K^*}(\lambda^0)$,\;
  $H(x^0, \lambda^0)$,\; and\;
  $\theta(x^0, \lambda^0)
  = \tfrac{1}{2}\|H(x^0, \lambda^0)\|^2$
\While{$\|H(x^k, \lambda^k)\| \geq \texttt{tol}$ \textbf{and}
  $k < \texttt{maxiter}$}
  \State Compute $J_H(x^k, \lambda^k)$
  \State Solve\;
    $J_H(x^k, \lambda^k)\, (d_x^k, d_\lambda^k)
    = -H(x^k, \lambda^k)$
  \If{$\nabla \theta(x^k, \lambda^k)^T (d_x^k, d_\lambda^k)
    \geq 0$ \textbf{or}
    $\|(d_x^k, d_\lambda^k)\| < \texttt{dtol}$}
    \If{$\nabla \theta(x^k, \lambda^k) \neq 0$}
      \State Solve\;
        $\left(J_H(x^k, \lambda^k)^T J_H(x^k, \lambda^k)
        + \sqrt{\theta(x^k, \lambda^k)}\,\Id\right)
        (d_x^k, d_\lambda^k)
        = -\nabla \theta(x^k, \lambda^k)$
    \ElsIf{$\nabla \theta^{\mathrm{feas}}(x^k, \lambda^k)
      \neq 0$}
      \State Solve\;
        $\left(J_H(x^k, \lambda^k)^T J_H(x^k, \lambda^k)
        + \sqrt{\theta(x^k, \lambda^k)}\,\Id\right)
        (d_x^k, d_\lambda^k)
        = -\nabla \theta^{\mathrm{feas}}(x^k, \lambda^k)$
    \Else
      \State Strongly stationary point found. \textbf{Terminate.}
    \EndIf
  \EndIf
  \State $\alpha_k \gets 1$,\; $\ell \gets 0$
  \While{$\theta(x^k + \alpha_k\, d_x^k,\,
    \lambda^k + \alpha_k\, d_\lambda^k)
    > \theta(x^k, \lambda^k)
    + c\, \alpha_k\, \nabla \theta(x^k, \lambda^k)^T
      (d_x^k, d_\lambda^k)$
    \textbf{and} $\ell < \texttt{maxiter\_ls}$}
    \State $\alpha_k \gets \alpha_k / 2$
    \State $\ell \gets \ell + 1$
  \EndWhile
  \State $x^{k+1} \gets x^k + \alpha_k\, d_x^k$
  \State $\lambda^{k+1} \gets \lambda^k
    + \alpha_k\, d_\lambda^k$
  \State Compute $\Pi_{\K^*}(\lambda^{k+1})$,\;
    $H(x^{k+1}, \lambda^{k+1})$,\; and\;
    $\theta(x^{k+1}, \lambda^{k+1})$
  \State $k \gets k + 1$
\EndWhile
\end{algorithmic}
\end{algorithm}

\section{Numerical Experiments} \label{section_num_exps}

We evaluate the performance of the semi-smooth Newton method (SSN) on
instances of the circular cone programming and low-rank matrix
completion problems. The experiments in
Subsection~\ref{circ_cone_numexp} were conducted using MATLAB R2022b
on an Intel Core i7-8700 CPU @ 3.20\,GHz with 16\,GB of RAM, and
those in Subsection~\ref{lrank_prob} on an Intel Core i9-12900K CPU @ 3.20\,GHz with 128\,GB of RAM. Tolerances and parameters specific to
each problem are provided in the corresponding subsections.

\subsection{Circular cone programming}\label{circ_cone_numexp}

Circular cone programming is an active area of research, with various
smoothing methods proposed for solving it; see, for example,
\cite{Chi:2016_4, Tang:2022_4}. This problem is a natural instance of
the generalized simplicial conic programming framework developed in
this paper: the circular cone $\LL^n_{\omega}$ is the image of the
second-order cone $\LL^n$ under a linear transformation, so the
projection machinery of Section~\ref{section_gensimpli_properties}
applies directly. The problem reads
\begin{equation}\label{circ_cone_prob}
\begin{array}{cl}
\min & c^T x \\
\text{s.t.} & Ax = b \\
            & x \in \LL^n_{\omega},
\end{array}
\end{equation}
where the circular cone of half-aperture $\omega \in (0, \tfrac{\pi}{2})$
is
\begin{equation}\label{circ_cone_def}
\LL^n_{\omega}
:= \left\{x = (x_1, u) \in \R \times \R^{n-1}
   : \|u\| \leq x_1 \tan \omega \right\}.
\end{equation}
Setting
\begin{equation}\label{circ_cone_M}
M = \begin{pmatrix}
\cot \omega & 0^T \\
0 & \Id
\end{pmatrix},
\end{equation}
we have $\LL^n_{\omega} = M\LL^n$. When $\omega = \frac{\pi}{4}$,
$\LL^n_{\omega}$ reduces to the standard second-order cone $\LL^n$.
The dual cone is
\begin{equation}\label{circ_cone_dual}
(\LL^n_{\omega})^*
= \left\{x = (x_1, u) \in \R \times \R^{n-1}
   : \|u\| \leq x_1 \cot \omega \right\}.
\end{equation}

Following \cite{Tang:2022_4}, we generate $10$ random instances for
each configuration and compare SSN with the smoothing Newton method of
\cite{Tang:2022_4}, hereafter denoted TZ. We test four half-aperture
angles $\omega \in \bigl\{\frac{\pi}{12},\, \frac{\pi}{6},\,
\frac{\pi}{4},\, \frac{\pi}{3}\bigr\}$.
The parameters for TZ are $\lambda = 0.5$, $\delta = 0.8$,
$\mu_0 = 10^{-2}$, $\gamma = 10^{-3}$, $\eta_k = 0.95^k$; for SSN
we set $c = 0.1$. Both methods use $\texttt{maxiter} = 100$,
$\texttt{maxiter\_ls} = 20$, and $\texttt{tol} = 10^{-8}$. For
$n = 1000$ both methods are compared; for $n = 5000$ and $n = 7500$
only SSN is run, as the execution times of TZ become prohibitive.

We test four starting points:
\begin{itemize}
\item \textbf{SP0:} $x^0 = 0$, $\lambda^0 = 0$;
\item \textbf{SP1:} $x^0 \in \operatorname{int}(\LL^n_{\omega})$,
  $\lambda^0 = 0$;
\item \textbf{SP2:} $x^0 = e_1$ (first canonical vector),
  $\lambda^0 = \mathbf{1}$ (vector of ones);
\item \textbf{SP3:} $x^0$ and $\lambda^0$ drawn uniformly at random.
\end{itemize}
In the tables that follow, each row corresponds to a problem instance
and each column pair to an angle, with the SSN result on the left and
the TZ result on the right. 

\begin{table}[H]
\centering
\resizebox{\textwidth}{!}{\begin{tabular}{|ll|ll|ll|ll|}
\hline
\multicolumn{2}{|c|}{$\omega = \frac{\pi}{12}$} & \multicolumn{2}{c|}{$\omega = \frac{\pi}{6}$} & \multicolumn{2}{c|}{$\omega = \frac{\pi}{4}$} & \multicolumn{2}{c|}{$\omega = \frac{\pi}{3}$}
\\\hline
\multicolumn{2}{|c|}{SSN - TZ} & \multicolumn{2}{c|}{SSN - TZ} & \multicolumn{2}{c|}{SSN - TZ} & \multicolumn{2}{c|}{SSN - TZ}
\\\hline
2.4955e-09 & 1.4267e-07 & 1.76e-09   & 2.521e-12  & 1.4894e-12 & 1.4853e-12 & 1.3743e-10 & 2.0924e-11
 \\
9.1313e-09 & 2.4828     & 2.5711e-09 & 2.3644e-12 & 1.473e-12  & 1.4375e-12 & 1.8208e-12 & 1.4809e-11
 \\
3.1483e-09 & 1.5277e-07 & 1.5646e-12 & 3.9581e-12 & 1.3539e-12 & 1.4404e-12 & 7.5853e-12 & 5.755e-11
 \\
9.0684e-10 & 1.7432e-12 & 1.744e-12  & 6.2934e-12 & 1.5753e-12 & 1.4228e-12 & 3.4163e-11 & 1.5875e-10
 \\
1.976e-09  & 1.724e-12  & 1.4995e-12 & 3.0344e-12 & 1.4763e-12 & 1.4707e-12 & 3.2682e-12 & 1.7282e-11 
 \\
9.0142e-10 & 1.7609e-12 & 1.5669e-12 & 3.147e-12  & 1.4123e-12 & 1.4878e-12 & 2.4459e-10 & 2.4211e-11 
\\
5.8438e-09 & 2.0803e-07 & 1.5131e-12 & 3.5226e-11 & 1.4924e-12 & 1.4186e-12 & 8.7299e-11 & 5.0756e-12
 \\
2.8914e-09 & 1.8417e-12 & 1.7722e-12 & 1.0519e-11 & 1.5216e-12 & 1.567e-12  & 1.5005e-12 & 6.1043e-12
 \\
2.8631e-09 & 1.8216e-12 & 1.5705e-12 & 4.0537e-12 & 1.8327e-12 & 1.5473e-12 & 1.0723e-09 & 2.498e-10
 \\
1.5002e-09 & 0.0010149  & 1.8707e-12 & 5.0224e-12 & 1.5909e-12 & 1.5174e-12 & 4.5275e-10 & 5.0215e-11
 \\\hline
\end{tabular}}
\caption{Final residuals $\|H(x,\lambda)\|$ for SP0,
  $n = 1000$.}\label{table_resSP0n1000}
\end{table}

\begin{table}[H]
\centering
\resizebox{0.75\textwidth}{!}{\begin{tabular}{| c c | c c | c c | c c | c c |}
\hline
\multicolumn{2}{|c|}{$\omega = \frac{\pi}{12}$} & \multicolumn{2}{c|}{$\omega = \frac{\pi}{6}$} & \multicolumn{2}{c|}{$\omega = \frac{\pi}{4}$} & \multicolumn{2}{c|}{$\omega = \frac{\pi}{3}$}
\\\hline
\multicolumn{2}{|c|}{SSN - TZ} & \multicolumn{2}{c|}{SSN - TZ} & \multicolumn{2}{c|}{SSN - TZ} & \multicolumn{2}{c|}{SSN - TZ}
\\\hline
1.0785 & 169.84 & 1.0446 & 102.34 & 0.88914 & 45.077 & 1.1596 & 1.6865
 \\
1.0415 & 177.16 & 1.0296 & 104.05 & 0.89731 & 47.025 & 1.0635 & 1.6968
 \\
1.0281 & 172.22 & 1.0072 & 104.81 & 0.86835 & 48.52  & 1.0699 & 1.6405 
 \\
1.0514 & 169.39 & 1.0422 & 108.36 & 0.94868 & 50.912 & 1.0326 & 1.7228
 \\
1.0542 & 168.38 & 1.1418 & 105.37 & 0.89336 & 47.222 & 1.0645 & 1.649 
 \\
1.0847 & 168.37 & 1.0181 & 102.05 & 0.88969 & 51.721 & 1.1554 & 1.6744
\\
1.0634 & 173.52 & 1.036  & 103.98 & 1.2417  & 51.976 & 1.1675 & 1.6994
\\
1.0221 & 169.38 & 1.0185 & 102.68 & 0.91758 & 49.256 & 1.1858 & 1.9668
 \\
1.0259 & 167.09 & 1.0573 & 104.66 & 0.88781 & 46.547 & 1.206  & 1.6999
 \\
1.0396 & 172.48 & 1.0945 & 105.72 & 0.86437 & 48.189 & 1.2136 & 1.7597
 \\\hline
\end{tabular}}
\caption{Computation times (seconds) for SP0,
  $n = 1000$.}\label{table_timesSP0n1000}
\end{table}

Table~\ref{table_resSP0n1000} reports the final residuals
$\|H(x, \lambda)\|$ for starting point SP0. Both methods achieve the
desired precision for most angles; however, for
$\omega = \frac{\pi}{12}$, SSN consistently produces smaller
residuals. The computation times in
Table~\ref{table_timesSP0n1000} reveal a substantial advantage for
SSN: for $\omega = \frac{\pi}{6}$ and $\frac{\pi}{4}$, SSN is
approximately two orders of magnitude faster. Even for
$\omega = \frac{\pi}{3}$, where both methods converge quickly, SSN
remains faster. Regarding iteration counts, SSN required at most $8$
iterations for every instance where it converged, whereas TZ required
up to $65$, $34$, and $7$ iterations for
$\omega = \frac{\pi}{6}$, $\frac{\pi}{4}$, and $\frac{\pi}{3}$,
respectively.

We also ran the same instances starting from SP1. For angles
$\omega = \frac{\pi}{12}$, $\frac{\pi}{6}$, and $\frac{\pi}{4}$, SSN
encountered strongly stationary points with
$\nabla \theta(x, \lambda) = 0$ but $\theta(x, \lambda) \neq 0$, and
thus failed to find optimal solutions, while TZ converged
successfully. For $\omega = \frac{\pi}{3}$, both methods converged,
with SSN being faster.

For $n = 5000$ with starting point SP0, the results were
satisfactory: the average computation times were approximately $58$, $48$,
$43$, and $58$ seconds for
$\omega = \frac{\pi}{12}$, $\frac{\pi}{6}$, $\frac{\pi}{4}$, and
$\frac{\pi}{3}$, respectively. Residuals were of order $10^{-11}$ for
the first three angles (with one exception at $10^{-8}$) and
$10^{-10}$ for $\omega = \frac{\pi}{3}$. For $n = 7500$ with SP0, the
average residual was $3.7 \times 10^{-10}$ and computation times were
approximately $186$, $162$, $138$, and $174$ seconds for each angle. The
behavior for other starting points was consistent with the $n = 1000$
case: SP1 led to strongly stationary points, while SP2 and SP3
performed well.

We also run both methods with $\omega = \frac{5\pi}{12}$ and $n=1000$ with different starting points, however neither SSN nor TZ were able to converge within the iteration limit.

\begin{figure}[H]
\centering
\includegraphics[scale=0.33]{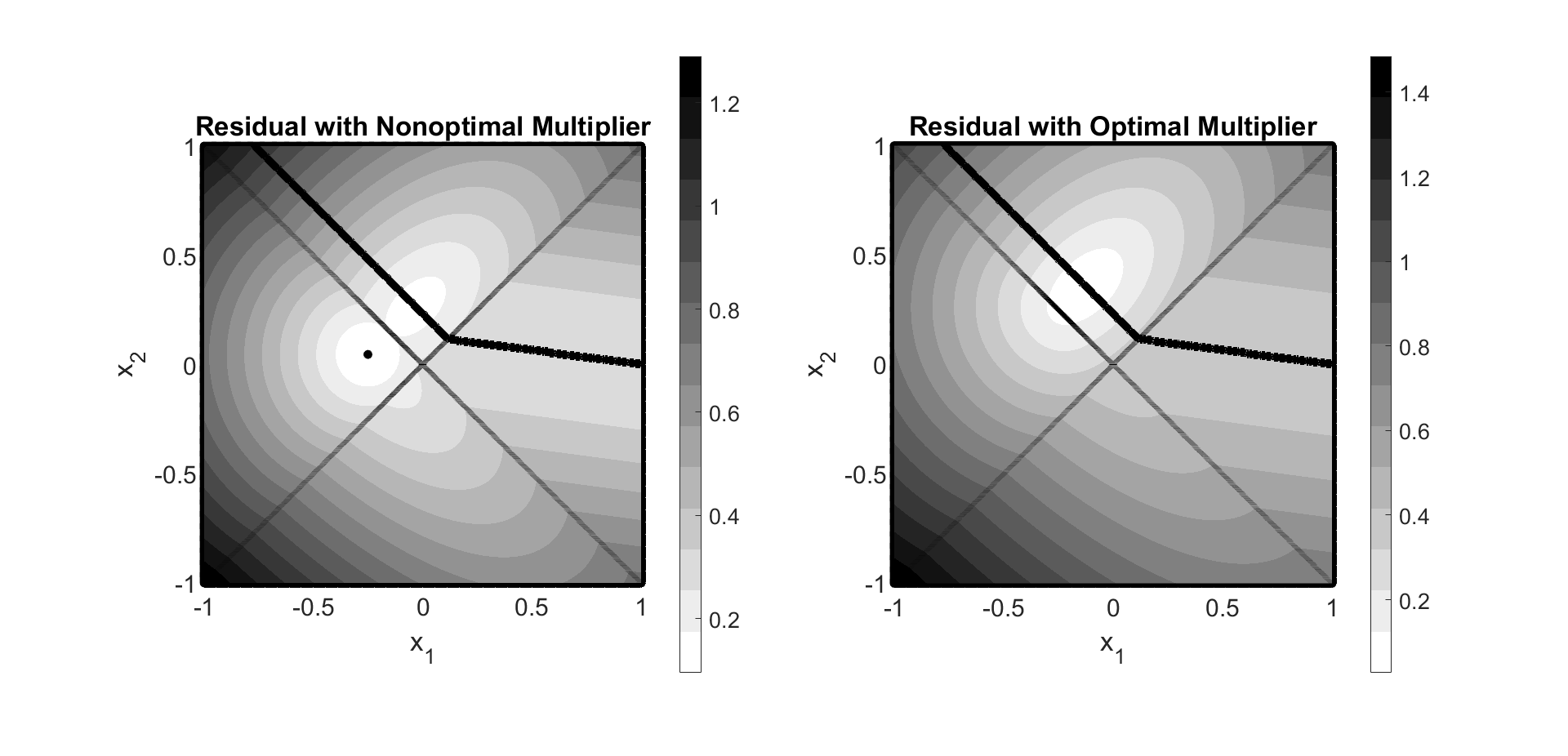}
\caption{Stationary point for $\theta(x,\lambda)$ with a non optimal and a optimal multiplier.}
\label{circularcone_nonopt_v_opt}
\end{figure}

We conclude this subsection with a remark on the role of strongly stationary points.
Figure~\ref{circularcone_nonopt_v_opt} illustrates how the landscape of stationary
points depends on the multiplier. For a second-order conic problem in dimension two,
we fix two multipliers: one optimal and one nonoptimal. The boundary of the
second-order cone $\LL^n$ is shown as a dark gray crossed line passing through the origin;
the sublevel sets of the residual
$\|H(x,\lambda)\|$ are displayed in shades of gray (decreasing from black to white),
and the points satisfying the feasibility condition
$A\Pi_{\LL^n}(x) - b = 0$ are shown in a continous black line.
The right plot illustrates how optimality and feasibility meet at a stationary point.
In contrast, the left panel shows that, with a nonoptimal multiplier,
the method can converge to a strongly stationary point that is not a solution
(isolated point highlighted in black). Developing strategies to detect and escape such points
remains an open question and could further improve the method’s performance.

\subsection{Low-rank matrix completion}\label{lrank_prob}

We now apply the semi-smooth Newton method to the low-rank matrix
completion problem. Finding high-precision solutions is challenging due
to the combinatorial nature of the rank constraint; significant efforts
have been made in \cite{Bertsimas:2021_4, Bertsimas:2023_4}. In
\cite{Bertsimas:2021_4}, a reformulation was introduced that replaces
the rank constraint with a matricial quadratic constraint and a trace
constraint via the Stiefel manifold, enabling the application of
continuous optimization methods. The original problem is
\begin{equation}\label{lowrnk_orig}
\begin{array}{cl}
\min & \dfrac{1}{2}\|H \circ X - G\|^2 \\[6pt]
\text{s.t.} & \operatorname{rank}(X) \leq \sigma,
\end{array}
\end{equation}
where $\sigma$ is a positive integer, $G \in \R^{n \times n}$ is the
data matrix, and $H \in \{0,1\}^{n \times n}$ is the mask matrix
encoding the observed entries: $H_{ij} = 1$ if $G_{ij}$ is known and
$H_{ij} = 0$ otherwise, with $\circ$ denoting the Hadamard (entrywise)
product. By \cite[Prop.~1]{Bertsimas:2021_4}, problem
\eqref{lowrnk_orig} is equivalent to
\begin{equation}\label{lowrnk_com_prob}
\begin{array}{cl}
\min & \dfrac{1}{2}\|H \circ X - G\|^2 \\[6pt]
\text{s.t.} & X - YX = 0 \\
            & Y^2 - Y = 0 \\
            & \operatorname{Tr}(Y) \leq \sigma \\
            & Y \in \sym^n,\quad X \in \R^{n \times n}.
\end{array}
\end{equation}

In \cite{Bertsimas:2021_4}, the authors address a regularized variant of problem
\eqref{lowrnk_com_prob} using two complementary approaches.
First, they solve a convex relaxation based
on semidefinite programming, which provides lower bounds and
high-quality approximate solutions. Second, they reformulate the problem
within their mixed-projection conic optimization framework and
solve it via a discrete optimization approach, namely an outer-approximation
algorithm embedded within a branch-and-bound scheme, yielding global
certifiable optimality.

There are two main differences between their approach and ours.
First, in \cite{Bertsimas:2021_4}, the problem is regularized with a Frobenius norm term,
which induces strong convexity in the primal variable and ensures
strong duality and dual attainment for the resulting convex subproblems
arising from the projection-based reformulation. This property is essential
for deriving the saddle-point representation exploited in their outer-approximation
framework. In contrast, we aim to solve the original problem \eqref{lowrnk_com_prob}
without such regularization. Second, due to the nature of the semismooth Newton
SSN method, our approach is designed to compute first-order
stationary points, rather than globally optimal solutions.
In particular, this difference suggests a possible hybrid strategy: the approximate
solutions obtained via the SDP relaxation in \cite{Bertsimas:2021_4}
can be used as warm starts, which may then be refined to high
precision using the SSN method.

For the tests we generate instances following \cite{Bertsimas:2021_4}. For each
combination of sparsity level $p \in \{10\%, 20\%, 30\%\}$ and rank
bound $\sigma \in \{1, 2, 3\}$, we solve $20$ random instances of
dimension $n = 10$, giving $180$ problems in total. SSN is terminated
when either the residual or the progress between consecutive iterates
falls below $10^{-8}$.

Since the sparsity levels are low and the rank constraint is tight, we
employ two initialization strategies. The first selects the row of the
incomplete matrix $G$ with the largest Frobenius norm and constructs a
rank-one initial matrix from it. The second generates perturbations of
$G$ with magnitudes depending on the sparsity level, aiming to explore
neighborhoods of different stationary points. For each instance, we
report the best outcome of the two strategies.

\begin{table}[ht]
\centering
\begin{tabular}{|c|c|c|c|c|c|c|}\hline
  & \multicolumn{3}{c|}{$\sigma = 1$}
  & \multicolumn{3}{c|}{$\sigma = 2$} \\
\cline{2-4}\cline{5-7}
$p$ & 0.1 & 0.2 & 0.3
    & 0.1 & 0.2 & 0.3 \\\hline
& $\leq$ 20s\;--\;100\%
& $\leq$ 32s\;--\;70\%
& $\leq$ 200s\;--\;70\%
& $\leq$ 20s\;--\;100\%
& $\leq$ 60s\;--\;75\%
& $\leq$ 500s\;--\;70\% \\
& --
& $\leq$ 500s\;--\;85\%
& $\leq$ 400s\;--\;85\%
& --
& $\leq$ 800s\;--\;90\%
& $\leq$ 6600s\;--\;85\% \\
& --
& $\leq$ 745s\;--\;90\%
& --
& --
& $\leq$ \(\text{T}_{\text{max}}\)\;--\;100\%
& $\leq$ \(\text{T}_{\text{max}}\)\;--\;95\% \\\hline
\end{tabular}

\medskip

\begin{tabular}{|c|c|c|c|}\hline
  & \multicolumn{3}{c|}{$\sigma = 3$} \\
\cline{2-4}
$p$ & 0.1 & 0.2 & 0.3 \\\hline
& $\leq$ 10s\;--\;90\%
& $\leq$ 121s\;--\;75\%
& $\leq$ 431s\;--\;55\% \\
& $\leq$ \(\text{T}_{\text{max}}\)\;--\;100\%
& $\leq$ 862s\;--\;90\%
& $\leq$ \(\text{T}_{\text{max}}\)\;--\;70\% \\
& --
& $\leq$ \(\text{T}_{\text{max}}\)\;--\;95\%
& -- \\\hline
\end{tabular}
\caption{Cumulative percentages of instances solved within given time
  thresholds, by rank bound $\sigma$ and sparsity level
  $p$.}\label{table_lowrank_times}
\end{table}

The average \(\text{T}_{\text{max}}\)
for \(\sigma = 2, 3\) was 10000s and 6500s, respectively. However,
as shown in Table~\ref{table_lowrank_times}, most of the problems were solved
in considerably less time.
In constrast in Table 5, from \cite{Bertsimas:2021_4}, for solving
the generated instances their method needed, averaging all problems,
times of the order of 2384s, 135226s and 130535s for \(\sigma = 1, 2, 3\),
respectively. In Table 5 of \cite{Bertsimas:2021_4}, the authors present a
more comprehensive set of experiments on problems of dimensions 20 and 30,
which they were able to solve globally with certifiable optimality.
In their numerical experiments, they also evaluate the scalability of
their SPD-based relaxation on problems with dimensions up to $n=600$,
obtaining good results.
Out of the $180$ instances tested, SSN found a stationary point in $169$,
a success rate of approximately $94$\%.
This is notable given the combinatorial complexity of the underlying problem.
As expected, the computation time increases with both the sparsity level
and the rank bound, but the method solved the majority of instances within a
reasonable time. Among the 11 unsolved instances, the method typically
stalled due to small progress in the residual, but still achieved
residuals on the order of $10^{-4}$, a reasonable approximation,
particularly since the residuals were not normalized.

\newpage
\section{Concluding Remarks} \label{section_conlud_remarks}

In this work we extended Robinson's normal equations to the nonlinear
conic programming setting, introducing what we call the \emph{conic
projection equations}, and established their equivalence with the
first-order KKT conditions of the problem. Building on this characterization, we proposed a semi-smooth Newton
method and proved its local
quadratic convergence, extending
previous results from quadratic linearly constrained problems to the
general nonlinear conic setting. In the context of generalized
simplicial cones, we established the strong semi-smoothness of the
projection operator, including the
case of rank-deficient linear mappings.

The method was tested on two classes of problems. For circular cone
programming, comparisons with the smoothing method of
\cite{Tang:2022_4} across different half-aperture angles showed that
SSN achieves comparable or better precision while being substantially
faster, particularly in higher dimensions. For the low-rank matrix
completion problem, SSN found high-precision KKT points for the
majority of instances, demonstrating that continuous
optimization methods based on the reformulation of
\cite{Bertsimas:2021_4} offer a viable alternative to combinatorial
approaches.

An important direction for future work is the development of
strategies to detect and escape strongly stationary points that are not
solutions, a phenomenon observed in some of our experiments. Addressing
this issue could further broaden the applicability and robustness of
the method.

\subsection*{Aknowledgments}
This work was supported in part by U.S. National Science Foundation NSF,
United States of America, under the
Grant DMS-2307328, by Conselho Nacional de Desenvolvimento Científico
e Tecnológico CNPq, Brazil, and by Fundação de Amparo à Pesquisa do Estado
de São Paulo FAPESP, Brazil, under the Grants 2019/13096-2 and 2018/24293-0.

\end{document}